\documentclass{amsart}
\usepackage{graphicx,amsfonts,amssymb,amsmath,amsthm,url,
  verbatim, mathtools}
\usepackage{url, hyperref}
\usepackage[normalem]{ulem}
\usepackage{tikz}
\usetikzlibrary{matrix}
\usepackage{enumitem} 
\usepackage{booktabs} 
\allowdisplaybreaks 

\newcommand{\N}{\mathbb{N}}
\newcommand{\Z}{\mathbb{Z}}
\newcommand{\Q}{\mathbb{Q}}
\newcommand{\R}{\mathbb{R}}
\newcommand{\C}{\mathbb{C}}
\newcommand{\A}{\mathbb{A}}

\newcommand{\tr}{\mathrm{tr}\,}
\renewcommand{\S}{\mathcal{S}}

\newcommand{\SL}{\mathrm{SL}}
\newcommand{\GL}{\mathrm{GL}}
\newcommand{\Sp}{\mathrm{Sp}}
\newcommand{\GSp}{\mathrm{GSp}} 
\newcommand{\G}{\Gamma}

\newcommand{\disc}{\mathrm{disc}\,}
\newcommand{\cont}{\mathrm{cont}\,}
\newcommand{\ord}{{\mathrm ord}}
\newcommand{\Cl}{\mathrm{Cl}}

\renewcommand{\o}{{\mathfrak{o}}}
\newcommand{\p}{\mathfrak{p}}
\renewcommand{\a}{\alpha}
\renewcommand{\b}{\beta}

\newcommand{\triv}{1}
\newcommand{\ti}{^{\times}}
\newcommand{\back}{\backslash}
\renewcommand{\(}{\left(} \renewcommand{\)}{\right)}
\renewcommand{\[}{\left[} 
\newcommand{\mat}[4]{{\setlength{\arraycolsep}{0.5mm}\left( \begin{array}{cc}#1&#2\\#3&#4\end{array}\right)}} 
\newcommand{\T}[1]{\,{}^t\! {{#1}}} 
\newcommand{\hsp}[1]{\hspace{-#1 cm}}

\newcommand{\Norm}{{\rm N}_{L/F}}

\theoremstyle{plain}
\newtheorem{lemma}{Lemma}[section]
\newtheorem{thm}{Theorem}[section]
\newtheorem*{theorem}{Theorem}
\newtheorem{cor}{Corollary}[section]
\newtheorem{prop}[lemma]{Proposition}

\theoremstyle{definition}
\newtheorem{dfn}{Definition}[section]

\newtheorem{mytable}{Table}

\theoremstyle{remark}
\newtheorem*{rem}{Remark}

\begin{document}
\author{Jolanta Marzec}
\address{Department of Mathematics \\ Technische Universit\"at Darmstadt \\ Germany \newline\indent
Institute of Mathematics \\ University of Silesia\\
  Katowice, Poland} \email{marzec@mathematik.tu-darmstadt.de}

\title{Maass relations for Saito-Kurokawa lifts of higher levels}

\begin{abstract}
It is known that among Siegel modular forms of degree $2$ and level $1$ the only functions that violate the Ramanujan conjecture are Saito-Kurokawa lifts of modular forms of level $1$. These are precisely the functions whose Fourier coefficients satisfy Maass relations. More generally, the Ramanujan conjecture for $\mathrm{GSp}_4$ is predicted to fail only in case of CAP representations. It is not known though whether the associated Siegel modular forms (of various levels) still satisfy a version of Maass relations.  
We show that this is indeed the case for the ones related to P-CAP representations. 
Our method generalizes an approach of Pitale, Saha and Schmidt who employed representation-theoretic techniques to (re)prove this statement in case of level $1$. In particular, we compute and express certain values of a global Bessel period in terms of Fourier coefficients of the associated Siegel modular form. Moreover, we derive a local-global relation satisfied by Bessel periods, which allows us to combine those computations with a characterization of local components of CAP representations.
\end{abstract}

\maketitle

\section{Introduction}
It is known that a Siegel modular form $F$ is a (classical) Saito-Kurokawa lift of an elliptic modular form $f$ if and only if its Fourier coefficients satisfy the \emph{Maass relations}
$$a(F,\mat{a}{b/2}{b/2}{c}) =\sum_{r|\gcd(a,b,c)} r^{k-1} a(F,\mat{ac\over r^2}{b\over 2r}{b\over 2r}{1})\, .$$
The classical cuspidal Saito-Kurokawa lift of weight $k$ is a lift from a cuspidal modular form $f\in S_{2k-2}^{(1)}(\SL_2(\Z))$ with $k$ even; it is a cuspidal Siegel modular form $F\in S_k^{(2)}(\Sp_4(\Z))$. The first construction of such a lift was given by Maass in \cite{maass1979} using correspondences between Siegel and classical modular forms, Jacobi forms and modular forms of half-integral weight (see also \cite{eichzag}). However, Saito-Kurokawa lifts can be also constructed using representation theory (\cite{ps-sk}, \cite{sch-funct}). The advantage of the latter is that it can be easily generalised to lifts of modular forms of higher level, and also with an odd weight. In this case, if $k$ is even\footnote{If $k$ is odd, the construction leads to a non-holomorphic function (cf. \cite{PSSc1}).}, for $f\in S_{2k-2}^{(1)}(\G_0(N))$ with $N$ square-free there exists a cuspidal Siegel modular form of weight $k$ invariant under the action of a congruence subgroup of $\GSp_4(\Z)$ such that its spinor $L$-function is given by 
$$L(s,F)=C(s)L(s,f)\zeta(s-k+1)\zeta(s-k+2),$$
where $C(s)$ comes from a contribution at $p=\infty$ and $p|N$ (for a precise statement see \cite[Theorem 5.2]{schsaito}\footnote{
For us, $L(s,F)$ denotes a classical spinor $L$-function associated to $F$, and not to an automorphic representation corresponding to $F$; hence the shift in the argument in comparison to \cite{schsaito}.}).
This does not tell us though anything about the coefficients of $F$ and whether they satisfy similar Maass relations. Pitale, Saha and Schmidt showed in \cite{PSSc1} that this is indeed the case if $F\in S_k^{(2)}(\Sp_4(\Z))$ is a Hecke eigenform. 

From a representation theoretic point of view, a Saito-Kurokawa lift produces from a cuspidal automorphic representation $\pi$ of $\mathrm{PGL}_2(\A)$ a cuspidal automorphic representation $\Pi$ of $\mathrm{PGSp}_4(\A)$; we can think of $f$ and $F$ as vectors of matching weight in the vector spaces of $\pi$ and $\Pi$. What is important is that any representation $\Pi$ we obtain via this (generalised) Saito-Kurokawa lifting is a CAP representation. 

More precisely, consider a cuspidal Siegel modular form $F$ of level $\G_0(N_1, N_2)$ (see notation below). We say that $F$ is associated to a CAP representation if the following are true.
\begin{enumerate}
\item[1)] The adelisation of $F$ gives rise to an irreducible automorophic representation $\Pi$ of $\GSp_4(\A)$.
\item[2)] The representation $\Pi$ is equivalent at almost all places to a constituent of a globally induced representation from a proper parabolic subgroup of $\GSp_4$.
\end{enumerate}
Furthermore, we say that $F$ is associated to a P-CAP representation if the proper parabolic subgroup above is the Siegel parabolic subgroup.  The classical Saito-Kurokawa lifts  correspond exactly to the P-CAP representations. 
It is known (see \cite{PSch-Ramanujan}, Corollary 4.5) that if $k\ge 3$, then $F\in S_k^{(2)}(\G_0(1,N))$ that is associated to a CAP representation  is automatically associated to a P-CAP representation. If $k=1$ or $2$, one also has CAP representations associated to other parabolics (the so-called B-CAP and Q-CAP representations). 

Note that the first condition above automatically implies that $F$ is an eigenform of the local Hecke algebra at all primes not dividing $N_2$. For general $N_1$, $N_2$, there is no known explicit construction that generalises the classical Saito-Kurokawa lifts and exhausts the set of all P-CAP $F$ of level $\G_0(N_1, N_2)$. It seems difficult then to directly prove Maass relations from construction. Even for $N_1=1$ this is already reflected in a comment made by Ibukiyama in \cite{ibu-SKlift} after he provides the relations for a special subset of Fourier coefficients.
In this work we are able to derive Maass relations using methods of representation theory.

\begin{theorem} 
Let $N_1$, $N_2$ be positive integers, $N_1|N_2$, and let F be a cuspidal Siegel modular form of weight k and level $\G_0(N_1, N_2)$ that is associated to a P-CAP representation. Let $a, b, c$ be integers such that $\gcd(a,b,c,N_2)= 1$, $b^2-4ac<0$ and $\(\frac{b^2-4ac}{p}\) =-1$ for all $p|N_1$, and let $L$ be any positive integer whose all prime factors divide $N_2$. Then 
$$a(F, L\mat{a}{b/2}{b/2}{c}) = \sum_{r|\gcd(a,b,c)} r^{k-1}  a(F, L\mat{{ac\over r^2}}{{b\over 2r}}{{b\over 2r}}{1}).$$
\end{theorem}

In fact, theorem stated above is a corollary of a more general result (Theorem \ref{thm:Maass_relation}), where the condition on $\(\frac{b^2-4ac}{p}\)$ is replaced with an exact condition on a certain ray class group. The simplification presented above follows from a connection between equivalence classes of binary quadratic forms and elements of suitable ray class groups; this is the topic of section \ref{sec:ray_class_groups} and may be of independent interest.

A reader may find it surprising or unsatisfactory that, in contrary to the classical Maass relations, the right hand side of the above equality contains $L$. However, there are modular forms for which this cannot be changed. Indeed, as Schmidt showed in \cite{schsaito}, there exist paramodular forms that are Saito-Kurokawa lifts, and all their Fourier coefficients $a(F,T)$ have of course the $(2,2)$-entry of $T$ divisible by the level. 

We should also mention that Ibukiyama's relations alluded to above, when restricted to $F$ with a trivial character, are included in our result. In particular, one of the assumptions there was that $F\in S_k^{(2)}(\G_0(N))$ has a non-zero Fourier Jacobi coefficient of index $1$; this is not necessary in our construction.

Our work extends the method developed in \cite{PSSc1} to Saito-Kurokawa lifts of higher levels. It bases on the relation satisfied by vectors in local Bessel models for P-CAP representations (\cite{PSSc1}, Theorem 2.1) and the fact that certain values of a global Bessel period associated to a Siegel modular form $F$ can be expressed in terms of Fourier coefficients of $F$. We compute these values explicitly for Siegel modular forms invariant under the action of $\G_0(N_1,N_2)$ with $N_1|N_2$, and combine it with the local-global relation \eqref{Besselrelation} satisfied by Bessel periods to obtain a relation between Fourier coefficients (Theorem \ref{thm:the_relation}; generalisation of \cite[Theorem 2.10]{kst2}). Theorem \ref{thm:the_relation} is a basis for our main result.

Another approach to obtain Maass relations for $F\in S_k^{(2)}(\G)$ for some congruence subgroup $\Gamma$ of $\Sp_4(\Q)$ might be to follow the aforementioned construction of Saito-Kurokawa lifting due to Piatetski-Shapiro, \cite{ps-sk}. This was done by Ichino \cite[Lemma 7.1]{ich} for $\Gamma =\Sp_4(\Z)$ and very recently extended by Chen \cite[Lemma 5.7]{ch} to $\Gamma =\G_0(1,N)$ with $N$ odd and square-free. In fact, their motivation to choose such an approach lied in establishing a formula for pullbacks of Saito-Kurokawa lifts and its application to algebraicity of central critical values of certain automorphic $L$-functions for $\Sp_2\times \GL_2$ and $\GL_3\times \GL_2$.

Throughout the paper we use the following notation.
\begin{itemize}
\item $\N, \Z, \Q, \R, \C$ stand for the natural, integer, rational, real and complex numbers respectively;
$\Q_p$ denotes the $p$-adic numbers and $\Z_p$ the $p$-adic integers, $\A$ stands for the adeles of $\Q$ and $\A_f:=\prod_{p<\infty}'\Q_p$ the finite adeles;\\
the set of invertible elements in a ring $R$ is denoted by $R\ti$;
\item $M_n$ denotes the set of $n\times n$ matrices, whose identity element is $1_n$, and $i_2:=i1_2$; 
we use the superscript $M_n^{\rm sym}$ for symmetric matrices, and $M_n^+$ for the matrices with positive determinant;\\
we distinguish a set
$$\mathcal{P}_n:=\{ T\in {1\over 2}M_n^{\rm sym}\(\Z\)\colon T\mbox{ half-integral and positive definite}\} ,$$
where \emph{half-integral} means that $T$ has integers on the diagonal;\\
$\T{T}$ is the transpose of $T$ and $\tr T$ the trace of $T$;
$$\cont \(\begin{smallmatrix} a & b/2\\ b/2 & c\end{smallmatrix}\):=\gcd(a,b,c)\, ,\quad\disc \(\begin{smallmatrix} a & b/2\\ b/2 & c\end{smallmatrix}\):=b^2-4ac$$
are the content and the discriminant of the matrix $\(\begin{smallmatrix} a & b/2\\ b/2 & c\end{smallmatrix}\)$;
\item The letter $G$ will always stand for the group $\GSp_4$ defined as follows: 
$$\GSp_4(\Q ):=\{ g\in \mathrm{GL}_4(\Q )\colon\T{g}\(\begin{smallmatrix}  &  & 1 &  \\  &  &  & 1 \\ -1 &  &  &  \\  & -1 &  &  \\ \end{smallmatrix}\) g=\mu(g)\(\begin{smallmatrix}  &  & 1 &  \\  &  &  & 1 \\ -1 &  &  &  \\  & -1 &  &  \\ \end{smallmatrix}\)\} ,$$where $\mu (g)\in\Q\ti$,
$\Sp_4(\Q ):=\{ g\in G(\Q)\colon\mu(g)=1\} ;$
and for $n_1\leq n_2$:
$$I(n_1,n_2):=G(\Z_p)\cap\(\begin{matrix} \Z_p & p^{n_1}\Z_p & \Z_p & \Z_p\\ \Z_p & \Z_p & \Z_p & \Z_p\\ p^{n_2}\Z_p & p^{n_2}\Z_p & \Z_p & \Z_p\\ p^{n_2}\Z_p & p^{n_2}\Z_p & p^{n_1}\Z_p & \Z_p\\\end{matrix}\) ,$$
\item For $N=\prod_p p^{n_p}$ and $N_1|N_2$ having similar prime factorizations we define
$$I_{N_1,N_2}:=\prod_{p<\infty} I(n_{1,p},n_{2,p}) ;$$
$$K^*_{N}:=\prod_{p<\infty} \{ g\in\GL_2(\Z_p)\colon g= \mat{*}{}{}{*} \bmod{p^{n_p}}\} ;$$
$$\hsp{0.3}\G_0(N_1,N_2):=G(\Q)\cap G(\R)^+I_{N_1,N_2}=\Sp_4(\Z)\cap\(\begin{matrix} \Z & N_1\Z & \Z & \Z\\ \Z & \Z &\Z & \Z\\ N_2\Z & N_2\Z & \Z & \Z\\ N_2\Z & N_2\Z & N_1\Z & \Z\end{matrix}\) ;$$
\item For $N\in\N$,
$$\G_0(N):=\SL_2(\Z )\cap\mat{\Z}{\Z}{N\Z}{\Z} ,\quad
\G^0(N):=\SL_2(\Z )\cap\mat{\Z}{N\Z}{\Z}{\Z} ;$$
\item For $N\in\N, X\in\Z$ we put
$$(X,N^{\infty}):=\prod_{p|N}p^{\ord_p X}\, ,\qquad\mbox{where}\quad\ord_p X:=\max\{ n\in\Z\colon p^n|X\} ,$$
and $N^{\infty}$ denotes a formal number such that $N^l|N^{\infty}$ for all $l\in\Z^+$;
$$p^k\parallel N \qquad\mbox{means}\qquad k=\ord_p N ;$$
$\( X\over p\)$ denotes the Legendre symbol.
\end{itemize}

\section*{Acknowledgements}
The work presented in this paper was carried out at the University of Bristol and represents a part of PhD thesis of the author. Her studies and research were possible thanks to a funding provided by EPSRC. The author would like to
thank her supervisor Abhishek Saha for guidance, support and patience in explaining various subtleties.
\section{Preliminaries}
Let $\pi=\otimes_p \pi_p$ be an irreducible automorphic cuspidal representation of $G(\A)$ with trivial central character and such that $\pi_{\infty} =\mathcal{E} (k,k)$, the lowest weight representation of scalar minimal K-type of weight $k$ (see \cite{pitsch2}). Let $\Phi$ be an automorphic form in the space of $\pi$ and let $\phi_{\infty}$ be a lowest weight vector of $\pi_{\infty}$. This means that
\begin{equation}\label{eq:weight}
\Phi(g k_{\infty})=j(k_{\infty},i_2)^{-k}\Phi (g)\qquad\mbox{for all}\quad k_{\infty}\in K_{\infty},\, g\in G(\A)\, , 
\end{equation}
where 
\begin{itemize}
\item[(i)] $K_{\infty}$ is a maximal compact subgroup of $G(\R)^+$ such that any $k_{\infty}$ in $K_{\infty}$ is of the form $\(\begin{smallmatrix} A & B\\ -B & A \end{smallmatrix}\)$,
\item[(ii)] $j(\(\begin{smallmatrix} A & B\\ C & D \end{smallmatrix}\),i_2)=\det (Ci_2+D)$.
\end{itemize}
Furthermore, assume that there are integers $N_1, N_2$ with $N_1 | N_2$ such that $\Phi$ is right invariant by $I_{N_1,N_2}$. 

Define 
\begin{equation}\label{deadelisation}
F(Z):=\Phi (g) j(g,i_2)^k\mu(g)^{-k}\, ,
\end{equation}
where $g\in G(\R )$ is such that $g\cdot i_2=Z$ and 
$$\(\begin{smallmatrix} A & B\\ C & D \end{smallmatrix}\)\cdot i_2:=(Ai_2+B)(Ci_2+D)^{-1}\, .$$ 
Such a function $F$ is holomorphic and satisfies 
$$\forall_{\gamma=\(\begin{smallmatrix} A & B\\ C & D \end{smallmatrix}\)\in \G_0(N_1,N_2)}\, F|_k \gamma (Z):=\mu(\gamma)^k j(\gamma ,Z)^{-k}F(\gamma\cdot Z)=F(Z)\, ;$$
it is a cuspidal \emph{Siegel modular form} of degree $2$, level $\G_0(N_1,N_2)$ and weight $k$ that is an eigenform of the local Hecke algebra at all primes $p\nmid N_2$.\footnote{In fact, these functions, coming from irreducible representations, span the space of Siegel cusp forms of degree $2$, level $\G_0(N_1,N_2)$ and weight $k$. 
}
It follows that $F$ admits a unique Fourier expansion
$$F(Z)=\sum_{T\in\mathcal{P}_2} a(F,T)e(\tr (TZ)) ,\quad\mbox{where}\quad e(x):=e^{2\pi ix} ,$$
and its Fourier coefficients satisfy
\begin{equation}\label{eq:G^0(N)-equiv_of_Fourier_coeff}
a(F,T)=a(F,\T{A}TA)\qquad \mbox{for every } A\in\G^0(N_1).
\end{equation}

Observe that the correspondence \eqref{deadelisation} is bijective in a sense that to any Siegel cusp form $F$ that satisfies the above conditions and gives rise to an irreducible representation we can attach an automorphic form $\Phi$ via
$$\Phi (g) :=F|_k g_{\infty}(i_2),\quad g=g_{\Q}g_{\infty}g_0\in G(\Q)G(\R)^+I_{N_1,N_2}=G(\A) ;$$
$\Phi$ is called the \emph{adelisation} of $F$.
\section{Local and global Bessel models}\label{loc-glob}
Throughout this section $F$ is a non-archimedean local field of characteristic zero, $\o$ its ring of integers, $\p$ the maximal ideal of $\o$, $\varpi$ a generator of $\p$, and $q$ the cardinality of the residue field $\o /\varpi\o$. For our global application we will only need $F=\Q_p$, but because of a deeper analysis of an ideal class group in the next section, we would like to describe a torus $T$ in more detail.
\subsection{Local Bessel models for $\GSp_4$}\label{sec:Bessel_intro}
We recall the definition of the Bessel model following the exposition of Furusawa \cite{fur} and Pitale, Schmidt \cite{PS1}. Let $S \in M_2(F)$ be a symmetric matrix such that $d=\disc S=-4 \det S\neq 0$. For
$$S =\begin{pmatrix}  a & b/2\\  b/2 & c\\\end{pmatrix}$$
we define the element
$$\xi = \xi_S = \begin{pmatrix} b/2 & c\\ -a & -b/2\\\end{pmatrix}$$
and denote by $F(\xi)$ a two-dimensional $F$-algebra generated by $1_2$ and $\xi$. Note that 
$$\xi^2 = \begin{pmatrix} \frac{d}{4} &\\& \frac{d}{4} \end{pmatrix} .$$ 
Depending whether $d$ is a square in $F\ti$ or not, $F(\xi)$ is isomorphic either to $L=F\oplus F$ or to the field $L=F(\sqrt{d})$ via
\begin{equation}\label{eq:xi_to_L}
x 1_2 + y\xi\, \longmapsto\begin{cases} 
x +y\frac{\sqrt{d}}{2} & d\not\in (F\ti )^2\\ 
(x+y\frac{\sqrt{d}}{2}, x-y\frac{\sqrt{d}}{2})& d\in (F\ti )^2\end{cases}\, .\end{equation}
The determinant map on $F(\xi)$ corresponds to the norm map on $L$, defined by $N_{L/F}(z)=z\bar{z}$, where $z\mapsto\bar{z}$ is the usual involution on $L$ fixing $F$. We define the Legendre symbol as
$$\( {L\over\p}\) =\begin{cases} -1 & \mbox{if } L/F\mbox{ is an unramified field extension}\\
0 & \mbox{if } L/F\mbox{ is a ramified field extension}\\
1 & \mbox{if } L=F\oplus F\end{cases}\, .$$
If $L$ is a field, denote by $\o_L,\p_L,\varpi_L$ the ring of integers, the maximal ideal of $\o_L$ and a fixed choice uniformizer in $\o_L$, correspondingly. If $L=F\oplus F$, let $\o_L=\o\oplus\o$, $\varpi_L=(\varpi,1)$. Define an ideal $\mathfrak{P}:=\p\o_L$ in $\o_L$; note that $\mathfrak{P}=\p_L$ is prime only if $\( {L\over\p}\) =-1$, otherwise $\mathfrak{P}=\p_L^2$ if $\( {L\over\p}\) =0$ and $\mathfrak{P}=\p\oplus\p$ if $\( {L\over\p}\) =1$.

We define a subgroup $T =T_S$ of $\GL_2$ by
$$T(F) = \{g \in \GL_2(F): \T{g}Sg =\det(g)S\}\, .$$
It is not hard to verify that $T(F) = F(\xi )\ti$, so that $T(F)\cong L\ti$. We identify $T(F)$ with $L\ti$ via \eqref{eq:xi_to_L}.
We can consider $T$ as a subgroup of $G$ via
$$T \ni g \longmapsto\begin{pmatrix} g & 0\\ 0 & \det g\cdot \T{g^{-1}}\end{pmatrix} \in G.$$
Let us denote by $U$ the subgroup of $G$ defined by
$$U = \{u(X) =\begin{pmatrix} 1_2 & X\\ 0 & 1_2\\\end{pmatrix}\;|\;\T{X} = X\} ,$$
and finally let $R$ be the subgroup of $G$ defined by $R=TU$.

We fix a non-trivial additive character $\psi$ of $F$ such that $\psi$ is trivial on $\o$, but non-trivial on $\p^{-1}$, and define the character $\theta =\theta_S$ on $U(F)$ by 
\begin{equation}\label{def:theta}
\theta(u(X)):=\psi(\tr(SX))\, . 
\end{equation}
Let $\Lambda$ be a character of $T(F)$ such that $\Lambda |_{F\ti} =1$. Denote by $\Lambda \otimes \theta$ the character of $R(F)$ defined by $(\Lambda \otimes \theta)(tu) =\Lambda(t)\theta(u)$ for $t\in T(F), u\in U(F).$

Let $\pi$ be an irreducible admissible representation of the group $G(F)$ with trivial central character. We say that such a $\pi$ has a \emph{local Bessel model} of type $(\Lambda, \theta)$ if $\pi$ is isomorphic to a subrepresentation of the space  of all locally constant functions $B$ on $G (F)$ satisfying the local Bessel transformation property 
$$B(rg)=(\Lambda\otimes\theta )(r)B(g) \mbox{ for all } r \in R(F) \mbox{ and } g \in G (F)\, .$$ 
It is known by \cite{nov}, \cite{prasadbighash} that if a local Bessel model exists, then it is unique. If the local Bessel model for $\pi$ exists, we denote it by $\mathcal{B}^{\pi}_{\Lambda, \theta}$. In this case, we fix a (unique up to scalar) isomorphism of representations $\pi\to\mathcal{B}^{\pi}_{\Lambda, \theta}$ and denote the image of any $\phi\in\pi$ by $B_{\phi}$.

In the Lemma below we explain how to switch between Bessel models defined with respect to different matrices $S$. Together with  \cite[Lemma 1.1]{PS1long} that will allow us to assume, without any loss of generality, that the entries $a,b,c$ and the discriminant $d=b^2-4ac$ of $S$ satisfy the following conditions:
\begin{eqnarray}\label{conditions_on_S}
\hsp{0.5}\bullet &&\hspace{-0.5cm} a,b\in\o\mbox{ and } c\in\o\ti .\nonumber\\
\hsp{0.5}\bullet &&\hspace{-0.5cm}\mbox{If } d\notin (F\ti)^2,\mbox{ then } d\mbox{ is a generator of the discriminant of } L/F.\\
\hsp{0.5}\bullet &&\hspace{-0.5cm}\mbox{If } d\in (F\ti)^2,\mbox{ then } d\in\o\ti .\nonumber
\end{eqnarray}

\begin{lemma}
	Let $S\in M_2(F)$ be symmetric, and let $\Lambda$ be a character of the associated group $T_S(F)$. Let $A\in\GL_2(F)$ and $\alpha\in F\ti$. Let $S'=\alpha\,^t\! ASA$. Then $T_{S'}(F)=A^{-1}T_S(F)A$, so that
	$$\Lambda'(t')=\Lambda(At'A^{-1}),\qquad t'\in T_{S'}(F)\, ,$$
	defines a character of $T_{S'}(F)$. Let $\pi$ be an irreducible admissible representation of $G(F)$. Then $\pi$ has a local Bessel model of type $(\Lambda, \theta_S)$ if and only if it has a local Bessel model of type $(\Lambda', \theta_{S'})$.
\end{lemma}
\begin{proof} Indeed, if $B\in\mathcal{B}^{\pi}_{\Lambda, \theta}$, then $B'(g):=B(\(\begin{smallmatrix} A & \\ & \alpha^{-1}\T{A^{-1}}\end{smallmatrix}\) g)$, $g\in G(F)$ satisfies the Bessel transformation property and the map $B\to B'$ gives rise to a local Bessel model $\mathcal{B}^{\pi}_{\Lambda ', \theta_{S'}}$. 
\end{proof}

\subsection{Sugano's formula}\label{sec:Sugano}
We now investigate more closely the case when $\pi$ is \emph{spherical}, that is, $\pi$ has a non-zero $G(\o)$-invariant vector. Such a representation is a constituent of a representation parabolically induced from an unramified character $\gamma$ of the Borel subgroup of $G(F)$. The values of the character $\gamma$ at the matrices
$$\(\begin{smallmatrix} \varpi & & & \\ & \varpi & & \\ & & 1 & \\ & & & 1 \end{smallmatrix}\)\, ,\(\begin{smallmatrix} \varpi & & & \\ & 1 & & \\ & & 1 & \\ & & & \varpi \end{smallmatrix}\)\, ,\(\begin{smallmatrix} 1 & & & \\ & 1 & & \\ & & \varpi & \\ & & & \varpi \end{smallmatrix}\)\, ,\(\begin{smallmatrix} 1 & & & \\ & \varpi & & \\ & & \varpi & \\ & & & 1 \end{smallmatrix}\)$$
are called the \emph{Satake parameters} of $\pi$ and determine the isomorphism class of $\pi$. Because central character of $\pi$ is trivial, we can call them in turn $\alpha ,\beta ,\alpha^{-1} ,\beta^{-1}$.

Throughout this section we assume the following:
\begin{enumerate}[label=(\roman*)]
	\item\label{as:pi} $\pi$ is a spherical representation of $G(F)$,
	\item\label{as:S} $S =\(\begin{smallmatrix}  a & b/2\\ b/2 & c\end{smallmatrix}\)$ with $a,b,c$ satisfying the conditions \eqref{conditions_on_S},
	\item\label{as:theta} $\theta=\theta_S$ is the character of $U(F)$ as in \eqref{def:theta},
	\item\label{as:Lambda} $\Lambda$ is a character of $T(F)$ that is invariant under the subgroup
	$$T(n):=T(F) \cap \{g \in \GL_2(\o): g= \mat{\lambda}{}{}{\lambda} \bmod{\p^n},\, \lambda\in\o\ti \}\, ,$$ 
	for some non-negative integer $n$.
\end{enumerate}

The next Lemma shows equivalent ways of writing the group $T(n)$. Thanks to this, our definition coincides with the one used in \cite{PSSc1} and \cite{PS1}.

\begin{lemma}\label{lem:T_p-and-L}
	The group $T(n)$ defined above is isomorphic to each of the following:
	$$T(F) \cap \{g \in \GL_2(\o): g= \mat{*}{}{\ast}{*} \bmod{\p^n}\} ,$$
	$$T(F) \cap \{g \in \GL_2(\o): g= \mat{*}{}{}{*} \bmod{\p^n}\}$$
	and (under the isomorphism $T(F)\cong L\ti$)
	$$\o\ti\(1+\mathfrak{P}^n\)\cap\o_L\ti\, .$$
	Moreover, every character of $T(F)$ that is trivial on $\o\ti$ is trivial on $T(n)$ for $n$ big enough.
\end{lemma}
\begin{proof}
	By \cite[Lemma 3.1.1]{pitsch}, $\o_L=\o +\o\xi_0$, where 
	$$\xi_0=\begin{cases} \frac{-b+\sqrt{d}}{2} & \mbox{if } L \mbox{ is a field}\\
	\(\frac{-b+\sqrt{d}}{2} ,\frac{-b-\sqrt{d}}{2}\)& \mbox{if } L=F\oplus F\end{cases}\, .$$
	Therefore by the identification \eqref{eq:xi_to_L}, 
	$$\o_L=\{\mat{x}{yc}{-ya}{x-yb} :x,y\in\o\}\, .$$
	Hence, under the assumptions \eqref{conditions_on_S} and via the isomorphism $T(F)\cong L\ti$, the group $T(\o):=T(F)\cap\GL_2(\o)$ is isomorphic to $\o_L\ti$ and $T(F)\cap M_2(\o)\cong\o_L$. Moreover, for any $\lambda\in\o\ti$:
	$$\lambda(1+\mathfrak{P}^n)\cap\o_L\ti \cong\lambda(1_2+\p^n\o_L)\cap T(\o),$$
	and thus 
	\begin{equation}\label{eq:T_p-and-L}
	T(n)\cong\o\ti\(1+\mathfrak{P}^n\)\cap\o_L\ti\, .\tag{$\star$}
	\end{equation}
	
	Assume now that $g\in T(F)\cap\GL_2(\o)$ is congruent to $\(\begin{smallmatrix} * & \\ * & *\end{smallmatrix}\)\bmod\p^n$. We already know that $g$ must be of the form 
	$x 1_2+y\(\begin{smallmatrix} & c \\ -a & -b\end{smallmatrix}\)$ with $x,y\in\o$. However, because $c\in\o\ti$, we have $y\in\p^n\o$, and thus $g=\(\begin{smallmatrix} x & \\ & x\end{smallmatrix}\)\bmod \p^n$, which means that $g\in T(n)$. The other inclusions are clear.
	
	To prove the last assertion we use the isomorphism \eqref{eq:T_p-and-L}. Because $\o\ti$ is compact and $\{\lambda(1+\mathfrak{P}^n)\cap\o_L\ti: n\in\N\}$ gives a set of neighbourhoods of each $\lambda\in\o\ti$ in $\o_L\ti$, so if $\Lambda$ is trivial on $\o\ti$, it must be trivial on $\{ \lambda(1+\mathfrak{P}^n)\cap\o_L\ti\colon\lambda\in\o\ti\}$ for $n$ big enough. 
\end{proof}

\begin{dfn}\label{def:c(Lambda)}
	The smallest integer $n$ for which $\Lambda$ is $T(n)$-invariant or, equivalently,
	$$\min\{ n\geq 0:\Lambda|_{\o\ti(1+\mathfrak{P}^n)\cap\o_L\ti} =1\}$$
	will be denoted by $c(\Lambda)$. 
\end{dfn}
\begin{mytable}[source: \cite{PS1}; but see also 
\cite{RS}, Theorem 6.2.2]\label{table2} The Bessel models of the irreducible, admissible representations of $\GSp_4(F)$ that can be obtained via induction from the Borel subgroup. For a detailed description of the representations see \cite[Section 2.2]{NF}. The pairs of characters $(\chi_1,\chi_2)$ in the $L=F\oplus F$ column for types IIIb and IVc refer to the characters of $T=\{\mathrm{diag} (a,b,b,a):a,b\in F\ti\}$ given by $\mathrm{diag} (a,b,b,a)\mapsto\chi_1(a)\chi_2(b)$.
\end{mytable}\vspace{-0.3cm}
\small{
\arraycolsep=4pt 
\thickmuskip = 2mu 
\begin{equation*}
 \begin{array}{rlccc}
  \toprule
   \text{Type}&&\text{representation}&
  \multicolumn{2}{c}{(\Lambda,\theta)\text{-Bessel functional exists exactly for \ldots}}
  \\
  \cmidrule{4-5}
  &&&L=F\oplus F&L/F\text{ a field extension}\\
  \toprule
   {\rm I}&& \chi_1 \times \chi_2 \rtimes \sigma\ \mathrm{(irreducible)}&\text{all }\Lambda&\text{all }\Lambda\\
  \midrule
   {\rm II}&{\rm a}&\chi \mathrm{St}_{\GL_2}\rtimes\sigma&\text{all }\Lambda&\Lambda\neq(\chi\sigma)\circ\Norm\\
  \cmidrule{2-5}
   &{\rm b}&\chi \triv_{\GL_2}\rtimes\sigma&\Lambda=(\chi\sigma)\circ\Norm&\Lambda=(\chi\sigma)\circ\Norm\\
  \midrule
  {\rm III}&{\rm a}&\chi \rtimes \sigma \mathrm{St}_{\GSp_2}&\text{all }\Lambda
   &\text{all }\Lambda\\
  \cmidrule{2-5}
   &{\rm b}&\chi \rtimes \sigma \triv_{\GSp_2}&\Lambda\in\{(\chi\sigma,\sigma),(\sigma,\chi\sigma)\}&\text{---}
   \\
  \midrule
   {\rm IV}&{\rm a}&\sigma\mathrm{St}_{\GSp_4}&\text{all }\Lambda&
  \Lambda\neq\sigma\circ\Norm\\
  \cmidrule{2-5}
   &{\rm b}&L(\nu^2,\nu^{-1}\sigma\mathrm{St}_{\GSp_2})&\Lambda=\sigma\circ\Norm
   &\Lambda=\sigma\circ\Norm\\
  \cmidrule{2-5}
   &{\rm c}&L(\nu^{3/2}\mathrm{St}_{\GL_2},\nu^{-3/2}\sigma)&\Lambda=(\nu^{\pm1}\sigma,\nu^{\mp1}\sigma)&\text{---}
\\
  \cmidrule{2-5}
   &{\rm d}&\sigma\triv_{\GSp_4}&\text{---}&\text{---}\\
  \midrule
   {\rm V}&{\rm a}&\delta([\xi,\nu \xi], \nu^{-1/2} \sigma)&\text{all }\Lambda
   &\Lambda\neq\sigma\circ\Norm,\,\Lambda\neq(\xi\sigma)\circ\Norm\\
  \cmidrule{2-5}
   &{\rm b}&L(\nu^{1/2}\xi\mathrm{St}_{\GL_2},\nu^{-1/2} \sigma)&\Lambda=\sigma\circ\Norm
   &\Lambda=\sigma\circ\Norm,\,\Lambda\neq(\xi\sigma)\circ\Norm\\
  \cmidrule{2-5}
   &{\rm c}&L(\nu^{1/2}\xi\mathrm{St}_{\GL_2},\xi\nu^{-1/2}\sigma)&\Lambda=(\xi\sigma)\circ\Norm
   &\Lambda\neq\sigma\circ\Norm,\,\Lambda=(\xi\sigma)\circ\Norm\\
  \cmidrule{2-5}
   &{\rm d}&L(\nu\xi,\xi\rtimes\nu^{-1/2}\sigma)&\text{---}
   &\Lambda=\sigma\circ\Norm,\,\Lambda=(\xi\sigma)\circ\Norm\\
  \midrule
   {\rm VI}&{\rm a}&\tau(S,\nu^{-1/2}\sigma)&\text{all }\Lambda
   &\Lambda\neq\sigma\circ\Norm\\
  \cmidrule{2-5}
  &{\rm b}&\tau(T,\nu^{-1/2}\sigma)&\text{---}&\Lambda=\sigma\circ\Norm\\
  \cmidrule{2-5}
   &{\rm c}&L(\nu^{1/2}\mathrm{St}_{\GL_2},\nu^{-1/2}\sigma)&\Lambda=\sigma\circ\Norm&\text{---}
   \\
  \cmidrule{2-5}
   &{\rm d}&L(\nu,1_{F^\times}\rtimes\nu^{-1/2}\sigma)&
   \Lambda=\sigma\circ\Norm&\text{---}\\
  \bottomrule
	\end{array}
\end{equation*}
}\normalsize
\newline

Under the assumptions \ref{as:pi}-\ref{as:Lambda}, $\pi$ has a local Bessel model of type $(\Lambda, \theta)$ for $\Lambda$ as specified in Table \ref{table2} above. For example, if $\pi$ is an irreducible spherical principal series representation (type I), such a local Bessel model exists for \emph{all} $\Lambda$. This model contains a unique (up to multiples) $G(\o)$-invariant vector, which we will denote by $B^{(0)}_{\pi}(\Lambda, \theta)$ or by $B^{(0)}_{\pi}$. Thanks to Sugano, we have an explicit formula for the values of $B^{(0)}_{\pi}$ at
$$h(l,m):=\begin{pmatrix}\varpi^{l+2m}&&&\\&\varpi^{l+m}&&\\&&1&\\&&&\varpi^{m}
\end{pmatrix},\qquad l, m\in\Z, m \geq 0.$$

\begin{thm}[Sugano; \cite{sug}]\label{thm:Sugano} Assume \ref{as:pi}-\ref{as:Lambda}, and let $\Lambda$ be such that the local Bessel model $\mathcal{B}^{\pi}_{\Lambda, \theta}$ exists. Then 
\begin{enumerate}
\item $B^{(0)}_{\pi}(h(l,m)) = 0$ if $l<0$ or $m < c(\Lambda)$.
\item $B^{(0)}_{\pi}(h(0,c(\Lambda))) \neq 0.$
\end{enumerate}
Moreover, 
$$\sum_{l, m \ge 0}B^{(0)}_{\pi}(h(l,m))x^my^l = \frac{H(x,y)}{P(x)Q(y)} ,$$ where 
$P(x), Q(y)$ and $H(x,y)$ are (explicit) polynomials depending on $\a, \b, L, F, \Lambda$.
\end{thm}

Because of the above theorem, if the local Bessel model $\mathcal{B}^{\pi}_{\Lambda, \theta}$ exists, we can and will henceforth normalize $B^{(0)}_{\pi}$ so that $B^{(0)}_{\pi}(h(0,c(\Lambda))) = 1$. 

\subsection{Global Bessel models for $\GSp_4$} Let $d$ be a fundamental discriminant and 
$$S(d) =\mat{a}{b/2}{b/2}{c} := \begin{cases} \begin{pmatrix}
\frac{-d}{4} & 0\\
0 & 1\\\end{pmatrix} & \text{ if } d\equiv 0\pmod{4} \\[4ex]
\begin{pmatrix} \frac{1-d}{4} & \frac12\\\frac12 & 1\\
\end{pmatrix} & \text{ if } d\equiv 1\pmod{4}\end{cases}.$$
Let groups $T, U, R$ be as above and let $\psi$ be a fixed non-trivial character of $\A /\Q$. We define the character $\theta = \theta_S$ on $U(\A)$ by $\theta(u(X))=\psi(\tr(SX))$. Let $\Lambda$ be a character of $T(\A) / T(\Q)$ such that $\Lambda |_{\A^\times}= 1$, and denote by $\Lambda \otimes \theta$ the character of $R(\A)$.

Let $\pi$ be an irreducible automorphic cuspidal representation of $G(\A)$ with trivial central character and $V_\pi$ be its space of automorphic
forms. For $\Phi \in V_\pi$, we define a function $B_\Phi$ on $G(\A)$ by
$$B_\Phi(g) =\int_{\A^\times R(\Q)\back R(\A)} (\Lambda \otimes \theta)(r)^{-1}\Phi (rg)\,dr.$$

The $\C$-vector space of functions on $G(\A)$ spanned by $\{B_\Phi :\Phi \in V_\pi \}$ is called the \emph{global Bessel space} of type $(\Lambda, \theta)$ for $\pi$, and its vectors are called \emph{Bessel periods}; it is invariant under the regular action of $G(\A)$, and when the space is non-zero, the corresponding representation is a model of $\pi$, which we call a \emph{global Bessel model} of type $(\Lambda, \theta)$.

For $\Phi\in V_\pi$ and a symmetric matrix $S\in M_2(\Q)$, we define the Fourier coefficient
\begin{equation}\label{fouriercoefficientdefeq}
\Phi_{S,\psi}(g)=\int\limits_{M_2^{\rm sym}(\Q)\backslash M_2^{\rm sym}(\A)}\psi^{-1}({\tr}(SX))\Phi(\mat{1_2}{X}{}{1_2}g)\,dX,\qquad g\in G(\A).
\end{equation}
For brevity we will often shorten $\Phi_{S,\psi}$ to $\Phi_{S}$.\\

\section{Ray class groups and $\G^0(N)$-equivalence}\label{sec:ray_class_groups}
In the next section we will describe a relation between Fourier coefficients of Siegel modular forms of degree $2$ and values of Bessel periods at the matrices similar to $h_p(l,m)$. As we shall see, these values are equal to sums of Fourier coefficients indexed via elements of a ray class group. The question is whether all the coefficients of fixed content and discriminant, up to equivalence, occur in such sum. It will be the subject of this section.

Recall that each coefficient $a(F,T)$ may be characterised according to content and discriminant of $T$, and each discriminant may be written as $dM^2L^2$, where $d$ is a \emph{fundamental discriminant}\footnote{Recall that $d$ is a fundamental discriminant if $d$ is square-free and $d\equiv 1\pmod 4$ or $d=4d'$, $d'$ square-free and $d'\equiv 2,3\pmod 4$. Or, equivalently, if $d=1$ or $d$ is the discriminant of a quadratic number field.} and $L$ content of the matrix $T$. From now on $d$ will denote a negative fundamental discriminant. 

In view of the relation \eqref{eq:G^0(N)-equiv_of_Fourier_coeff} satisfied by Fourier coefficients of $F$, we define  the set 
$$H(dM^2,L;\G^0(N)):=\{ T\in \mathcal{P}_2\colon \disc T=dL^2M^2, \cont T =L\} /\sim\, ,$$
where 
$$T\sim T'\qquad\Longleftrightarrow\qquad T'=\T{A}TA\quad\mbox{for some } A\in\G^0(N).$$

It is well-known that when $M=N=1$, the set $H(d,L;\G^0(1))$ is isomorphic to the ideal class group of $\Q(\sqrt{d})$. As we shall see, when $M,N>1$ the situation is more complicated. In \cite{PSSc1}, Pitale, Saha and Schmidt found a bijection between $H(dM^2,L;\G^0(1))$ and a certain ray class group of $\Q(\sqrt{d})$, which we will call later $\Cl_d(M)$. We are going to extend their result to $N>1$.

\subsection{Construction of an endomorphism} 
Fix positive integers $M, N$ and a negative fundamental discriminant $d$. Let
$S(d)$ and $T=T_{S(d)}$ be as in section \ref{loc-glob}. 

\begin{dfn}\label{def:T_N} For $N=\prod_p p^{n_p}$ define
	$$T_N:=\prod_{p<\infty} T(n_p)\quad\mbox{and}\quad\Cl_d(N):=T(\A )/T(\Q )T(\R )T_N\, ,$$
	where $T(n_p)\subseteq T(\Q_p)$ is as in section \ref{sec:Sugano} and by $T(0)$ we mean the maximal compact subgroup $T(\Z_p):= T(\Q_p) \cap \GL_2(\Z_p)$ of $T(\Q_p)$.
\end{dfn}
\noindent Because of the isomorphism described in Lemma \ref{lem:T_p-and-L}, we may view $\Cl_d(N)$ as a \emph{ray class group} of $\Q(\sqrt{d})$.

Basing on the argument of \cite{PSSc1}, we will now describe a certain map from $\Cl_d(N')$ to $H(dM^2,L;\G^0(N))$, where $N'$ is any integer divisible by $MN$.

Let $c \in\Cl_d(N')$ and let $t_{c}\in T(\A )$ be a representative for $c$ such that $t_{c}\in\prod_{p<\infty} T(\Q_p)$. By strong approximation we can write $t_{c}=\gamma_{c} m_{c}\kappa_{c}$, where $\gamma_{c}\in \GL_2(\Q )$, $m_{c}\in\GL_2(\R )^+$ and $\kappa_{c}\in K^*_{N'}$. Also, denote by $(\gamma_{c})_f$ the finite part of $\gamma_{c}$ when considered as an element of $\GL_2(\A )$, thus we have the equality $(\gamma_{c})_f=\gamma_{c} m_{c}$, as elements of $\GL_2(\A )$. Let
\begin{equation}\label{S_c} S_{c}:=\det (\gamma_{c})^{-1} \T{\gamma_{c}} S(d)\gamma_{c}\, ;\end{equation}
it is a positive definite, half-integral, symmetric matrix of discriminant $d$ and content $1$ (cf. \cite{fur}, p. 209). Put 
\begin{equation}\label{def:phi_{L,M}}
\phi_{L,M}(c) = L\(\begin{smallmatrix} M &  \\  & 1\\\end{smallmatrix}\)S_{c}\(\begin{smallmatrix} M &  \\  & 1\\\end{smallmatrix}\)\, .
\end{equation} 
Then $\phi_{L,M}(c)$ is a matrix of discriminant $dM^2L^2$ and content $L$.

Note that the matrices $\phi_{L,M}(c)$ constructed above are not uniquely defined, as they depend on the choice of $t_c$ and $\kappa_{c}$. However, the definition is correct for $\G^0(N)$-equivalence classes.

\begin{prop}\label{phi_injective}
	Assume that $MN|N'$. Then the map $\tilde{\phi}_{L,M}=\tilde{\phi}_{L,M;N'}$ from $\Cl_d(N')$ to $H(dM^2,L;\G^0(N))$, sending $c$ to $\phi_{L,M}(c)$
	is well-defined. Moreover, if $N'=MN$, $\tilde{\phi}_{L,M;N'}$ is injective.
\end{prop}
\begin{proof}
	This follows almost immediately from the proof of Proposition 5.3, \cite{PSSc1}. The first part goes without any change. To show injectivity, it suffices to exchange a group $\SL_2(\Z )$ occurring in the second part of the proof with $\G^0(N)$. More precisely, if we assume that there exists a matrix $A\in\G^0(N)$ such that $\T{A}\phi_{L,M}(c_2)A=\phi_{L,M}(c_1)$, then $A$ must be, in fact, an element of $\G^0(N)\cap\G_0(M)$. Observe that $\T{R}S_{c_2}R=S_{c_1}$ for $R=\(\begin{smallmatrix} M & \\ & 1\end{smallmatrix}\) A \(\begin{smallmatrix} 1/M & \\ & 1\end{smallmatrix}\)\in\G^0(MN)$, and so if $\gamma_1, \gamma_2$ correspond to $S_{c_1}, S_{c_2}$ via \eqref{S_c}, then $\gamma_2 R\gamma_1^{-1}\in T(\Q )$. Therefore, if we take $t_1=\gamma_1\gamma_{1,\infty}^{-1}\kappa_1$ and $t_2=\gamma_2\gamma_{2,\infty}^{-1}\kappa_2$ as representatives in $\prod_{p<\infty} T(\Q_p)$ of $c_1$ and $c_2$, then 
	$$\gamma_2 R\gamma_1^{-1}\in T(\Q )\cap t_2\GL_2(\R )^+ K_{MN}^0t_1^{-1}\, ,$$
where 
$$K^0_{MN}:=\prod_{p<\infty} \{ g\in\GL_2(\Z_p)\colon g= \mat{*}{}{\ast}{*} \bmod{p^{\ord_p(MN)}}\}\, .$$
This means (recall Lemma \ref{lem:T_p-and-L}) that $t_1$ and $t_2$ represent the same element in $\Cl_d(N')$, provided $N'=MN$.
\end{proof}

\subsection{The image of $\tilde{\phi}_{L,M}$}
We start with the observation that the image of the map $\tilde\phi_{L,M}$ from $\Cl_d(N')$ to $H(dM^2,L;\G^0(N))$ constructed above does not depend on $N'$, but only on $MN$.

\begin{lemma}Let $MN |N_1' | N_2'$, and let $\rho\colon\Cl_d(N_2')\to\Cl_d(N_1')$ be the natural projection. Then the following diagram is commutative:
	\begin{center}
		{\tiny
			\begin{tikzpicture}
			\matrix (m) [matrix of math nodes,row sep=3em,column sep=4em,minimum width=2em]
			{
				\Cl_d(N_2') & & \Cl_d(N_1')\\
				& & \\
				& H(dM^2,L;\G^0(N)) & \\};
			\path[-stealth]
			(m-1-1) edge node [left] {$\tilde{\phi}_{L,M;N_2'}$} (m-3-2)
			edge node [above] {$\rho$} (m-1-3)
			(m-1-3) edge node [right] {$\tilde{\phi}_{L,M;N_1'}$} (m-3-2);
			\end{tikzpicture}
		}\end{center}
\end{lemma}
\begin{proof}
	This follows by construction. Let $c \in \Cl_d(N_1')$ and $c_1, c_2,\ldots ,c_t$ be the elements of $\Cl_d(N_2')$ that the map $\rho$ sends to $c$. Choose distinct $i,j\in\{ 1,2,\ldots ,t\}$. We will show that $\phi_{L,M}(c_i)$ and $\phi_{L,M}(c_j)$ are $\G^0(N)$-equivalent. For this it suffices to find $\gamma\in\G^0(N)$ such that $\gamma_{c_i}\(\begin{smallmatrix} M & \\ & 1\end{smallmatrix}\) =\gamma_{c_j}\(\begin{smallmatrix} M & \\ & 1\end{smallmatrix}\)\gamma$. Denote by $(\gamma_{c_i})_p$ the image of $\gamma_{c_i}$ in $\GL_2(\Z_p)$, when embedded diagonally. Since $c_i,c_j$ map to $c$, $(\gamma_{c_i})_p T(\ord_p(N_1'))=(\gamma_{c_j})_p T(\ord_p(N_1'))$ for all primes $p|N_1'$ or $p\nmid {N_2'\over N_1'}$. Hence, for each of those primes there exists $g_p\in T(\ord_p(N_1'))$ such that $(\gamma_{c_i})_p =(\gamma_{c_j})_p g_p$. Note that we can choose $\gamma_{c_i}$ and $\gamma_{c_j}$ in such a way that $\gamma_{c_i} T(\Q)T(\R)\prod_p T(\Z_p)=\gamma_{c_j} T(\Q)T(\R)\prod_p T(\Z_p)$. Hence, for primes $p| {N_2'\over N_1'}$, $g_p:=(\gamma_{c_j}^{-1})_p(\gamma_{c_i})_p$ is still in $T(\Z_p)$. This shows that $g:=\gamma_{c_j}^{-1}\gamma_{c_i}\in\G^0(MN)$. Now it's easy to check that $\gamma:=\(\begin{smallmatrix} 1/M & \\ & 1\end{smallmatrix}\) g \(\begin{smallmatrix} M & \\ & 1\end{smallmatrix}\)$ gives a desired $\G^0(N)$-equivalence.
\end{proof}
	
Put
$$H_1(dM^2,L;\G^0(N)):=\mathrm{im}\(\tilde{\phi}_{L,M}\colon\Cl_d(MN)\to H(dM^2,L;\G^0(N))\)\, .$$
	
\begin{rem}
	The map $\tilde{\phi}_{L,M}$ from $\Cl_d(MN)$ to $H_1(dM^2,L;\G^0(N))$ is a bijection. In this way the set $H_1(dM^2,L;\G^0(N))$ acquires a natural group structure that makes it isomorphic to $\Cl_d(MN)$.  Hence if $\Lambda$ is any character of $\Cl_d(MN)$, then we can naturally think of $\Lambda$ as a character of $H_1(dM^2,L;\G^0(N))$. 
		
	In the next section we will naturally encounter sums like 
	$$\sum_{c \in \Cl_d(N')} \Lambda^{-1}(c) a(F,\phi_{L,M}(c))$$ 
	for a character $\Lambda$ of $\Cl_d(N')$. Observe that, if we denote by $\rho$ a natural projection from $\Cl_d(N')$ to $\Cl_d(MN)$, we have the following useful fact:
	\begin{multline*}
		\sum_{c \in \Cl_d(N')} \Lambda^{-1}(c) a(F,\phi_{L,M}(c))=\sum_{c\in\Cl_d(MN)} a(F,\phi_{L,M}(c))\sum_{\substack{\tilde{c}\in\Cl_d(N')\\ \rho (\tilde{c})=c}}\Lambda^{-1}(\tilde{c})\\
		= \begin{cases} 0 &\text{ if } C(\Lambda) \nmid MN \\  \frac{|\Cl_d(N')|}{|\Cl_d(MN|)} \sum_{T \in H_1(dM^2,L;\G^0(N))} \Lambda^{-1}(T) a(F,T) &\text{ if } C(\Lambda) | MN, \end{cases}
	\end{multline*}
	where $C(\Lambda):=\prod_p p^{c(\Lambda_p)}$ is the smallest integer such that $\Lambda|_{T_{C(\Lambda)}}=1$.
\end{rem}
	
Let us now try to describe more accurately the set $H_1(dM^2,L;\G^0(N))$. This will give us information on the coefficients occurring in the sum above.

\begin{lemma}\label{Lemma_on_E(S')} Suppose that $S' = \(\begin{smallmatrix}a' & b'/2\\ b'/2 & c'\end{smallmatrix}\)$ is a matrix of discriminant $dM^2L^2$ and content $L$, $\xi_{S'}=\(\begin{smallmatrix} b'/2 & c'\\ -a' & -b'/2\end{smallmatrix}\)$. Let $E(S')$ be the subgroup of $\SL_2(\Z)$ defined as 
$$E(S'):= \{g \in \SL_2(\Z): g^t S'g = S'\}\, .$$ Then 
	\begin{enumerate}\thickmuskip = 1.5mu 
		\item If $d \neq -4, -3$, or if $M>1$, then $E(S') = \{\pm 1_2\}$.
		\item If $(d,M)=(-4,1)$, then   $E(S') = \{\pm 1_2, \pm\xi_{S'}\}$. 
		\item If $(d,M)=(-3,1)$, then $E(S') = \{\pm 1_2, \pm \( {1\over 2}1_2+\xi_{S'}\) , \pm \( -{1\over 2}1_2+\xi_{S'}\)\}$.
	\end{enumerate}
\end{lemma}
\begin{proof} 
	Note that $E(S')=\{ g\in T_{S'}(\Q)\cap\SL_2(\Z)\colon\det g=1\}$ and it doesn't depend on the content of $S'$. We may assume then that $L=1$. A discussion at the beginning of section \ref{sec:Bessel_intro} applies also when $\disc S'=dM^2$ and $F=\Q$, i.e. there is an identification
	$$T_{S'}(\Q )=\Q(\xi_{S'})\ti\ni x+y\xi_{S'}\longmapsto x+y{\sqrt{\det \xi_{S'}}\over 2} =x+y{M\sqrt{d}\over 2}\in\Q (\sqrt{d})\, .$$
	Therefore $E(S')$ corresponds to the units of the ring of integers of $\Q (\sqrt{d})$ of the form $x+y{M\sqrt{d}\over 2}$. It is easy to check that they are of the form proposed above.
\end{proof}

\begin{prop}\label{prop_G^0(N)_repr} Suppose that $S_1,\ldots ,S_t$ are matrices which comprise a complete set of (distinct) representatives for $H(dM^2,L;\G^0(1))$, and $A_1$,..., $A_r$ form a complete set of (distinct) representatives for $\SL_2(\Z)/\Gamma^0(N)$.
	\begin{enumerate}
		\item Assume that either $d \neq -4, -3$ or $M>1$. Then $A_i^t S_j A_i$ gives a complete set of distinct representatives for $H(dM^2,L;\G^0(N))$, i.e.
		$$\hsp{0.1} |H(dM^2,L;\G^0(N))| =tr= {|\Cl_d(1)|\over u(d)} MN \prod_{p|M} (1-p^{-1} \left(\frac{d}{p} \right))\prod_{p|N} (1+1/p) ,$$
		where $u(d)=1$ if $d \neq -4, -3$, and $u(-3)=3, u(-4)=2$.
		\item Assume that $(d,M)=(-4,1)$ and let $N=2^{n_0}p_1^{n_1}\ldots p_s^{n_s}$ be the prime decomposition of $N$. Then
		$$|H(-4,L;\G^0(N))| = {1\over 2} \( N\prod_{p|N} (1+1/p) +\mathcal{L}_{-4}\) ,$$
		where
		$$\mathcal{L}_{-4} =\begin{cases} 
		2^s & \mbox{if } n_0\leq1\mbox{ and }\,\forall_{i} p_i\equiv 1\, (\!\!\!\!\!\!\mod{4})\\
		0 & \mbox{otherwise.}\end{cases}$$
		\item Assume that $(d,M)=(-3,1)$ and let $N=3^{n_0}p_1^{n_1}\ldots p_s^{n_s}$ be the prime decomposition of $N$. Then
		$$|H(-3,L;\G^0(N))| = {1\over 3} \( N \prod_{p|N} (1+1/p)+2\mathcal{L}_{-3}\) ,$$
		where 
		$$\mathcal{L}_{-3} =\begin{cases} 
		2^s & \mbox{if } n_0\leq 1\mbox{ and }\,\forall_{i} p_i\equiv 1\, (\!\!\!\!\!\!\mod{6})\\
		0 & \mbox{otherwise.}\end{cases}$$
	\end{enumerate}
\end{prop}
\begin{proof}
	Recall (e.g. \cite[Proposition 2.5]{iwantopics} and \cite[Theorem 8.2]{Cohn}) that 
	$$r=N\prod_{p|N} (1+1/p)$$
	and
	$$t=|\Cl_d(M)|=\begin{cases} {|\Cl_d(1)|\over u(d)} M \prod_{p|M} (1-p^{-1} \left(\frac{d}{p} \right)) & \mbox{if } M>1\\ |\Cl_d(1)| & \mbox{if } M=1 \end{cases} .$$
	Each equivalence class in $H(dM^2,L;\G^0(1))$ (i.e. $j\in\{ 1,\ldots ,t\}$ is fixed) can be written as a union of sets 
	$\{\T{g}\T{A_i} S_jA_ig: g\in\G^0(N)\}$ 
	with $i\in\{ 1,\ldots ,r\}$. The question is whether they are all disjoint. Assume this is not the case for the sets corresponding to $i_1$ and $i_2$, i.e. assume there exists $g\in\G^0(N)$ such that $\T{A_{i_1}} S_jA_{i_1}=\T{g}\T{A_{i_2}} S_jA_{i_2}g$. Then $S_j=\T{(A_{i_2}gA_{i_1}^{-1})} S_jA_{i_2}gA_{i_1}^{-1}$, where $A_{i_2}gA_{i_1}^{-1}\in\SL_2(\Z)$. Hence, $A_{i_2}gA_{i_1}^{-1}\in E(S_j)$ and Lemma \ref{Lemma_on_E(S')} tells us precisely what these elements may be. The question is whether it does not imply $i_1=i_2$ and how often this is the case. 
	\begin{enumerate}
		\item If $d \neq -4, -3$, or if $M>1$, then $A_{i_2}g=\pm A_{i_1}$, and so $i_1=i_2$.
		\item If $(d,M)=(-4,1)$, then either $i_1=i_2$ as above, or $A_{i_2}g=\pm\xi_{S_j} A_{i_1}$.
		\item If $(d,M)=(-3,1)$, then we get two additional possibilities: either $A_{i_2}g=\pm ({1\over 2}1_2+\xi_{S_j})A_{i_1}$ or $A_{i_2}g=\pm (-{1\over 2}1_2+\xi_{S_j})A_{i_1}$.
	\end{enumerate}
	Let us check whether the remaining cases may happen when $i_1\neq i_2$. Without loss of generality we may assume that $L=1$. Observe that both $H(-4,1;\G^0(1))$ and $H(-3,1;\G^0(1))$ contain only one class, namely the one determined by $1_2$ and $\(\begin{smallmatrix} 1 & 1/2\\ 1/2& 1\end{smallmatrix}\)$ respectively. Indeed, each of the elements $\(\begin{smallmatrix} a' & b'/2\\ b'/2& c'\end{smallmatrix}\)$ in $H(dM^2,1;\G^0(1))$ can be written uniquely in a {\it reduced form}, that is with $|b'|\leq a'\leq c'$ (e.g. \cite{Cox}, Theorem 2.8). It is easy to see that if $M=1$ and $d=-4,-3$, the only matrix satisfying these conditions is $1_2$ and $\(\begin{smallmatrix} 1 & 1/2\\ 1/2& 1\end{smallmatrix}\)$ respectively. From this observation it also follows that $|\Cl_{-4}(1)|=|\Cl_{-3}(1)|=1.$
	
	Choose a set of representatives for $\SL_2(\Z)/\G^0(N)$ to be\footnote{For the proof that these indeed constitute representatives for $\SL_2(\Z)/\G^0(N)$, consult \cite{iwantopics}, Proposition 2.5.}
	$$\{ A_1,\ldots ,A_r\} =\{\(\begin{smallmatrix} * & u\\ * & v\end{smallmatrix}\)\in\SL_2(\Z):v|N, u\hspace{-0.3cm} \pmod{N/v} \}\, .$$ 
	Let $A_{i_1}=\(\begin{smallmatrix} * & u\\ * & v\end{smallmatrix}\)$ and $A_{i_2}=\(\begin{smallmatrix} * & u'\\ * & v'\end{smallmatrix}\)$. 
	
	Assume first that $(d,M)=(-4,1)$. Then $g=\pm A_{i_2}^{-1}\(\begin{smallmatrix} & 1\\ -1& \end{smallmatrix}\) A_{i_1}$ is in $\G^0(N)$ if and only if $N|uu'+vv'$. Since $\gcd(u,v)=1=\gcd(u',v')$ and $v,v'|N$, we must have $v=\gcd(u',N)$, $v'=\gcd(u,N)$ and $\gcd(v,v')=1$. Put $u=v'\underline{u}$ and $u'=v\underline{u'}$, so that ${N\over vv'}|\underline{u}\underline{u'}+1$. Note that $\underline{u'}$ is determined by $v,v',\underline{u}$. Moreover, under the assumption $\gcd(v,v')=1$, $i_1=i_2$ if and only if $v=v'=1$ and $u=u'$ satisfies $u^2 =-1\bmod{N}$.
	
	Hence, if we fix $u,v$, then $u',v'$ are uniquely determined and thus there are 
	$$\frac{N}{2}\prod_{p|N} (1+1/p) +\frac{\mathcal{L}_{-4}}{2}$$
	$\G^0(N)$-non-equivalent classes within each class in $H(-4,L;\G^0(1))$, where
	$$\mathcal{L}_{-4} :=\#\{ u\in\{ 1,\ldots ,N\} :u^2 =-1\bmod{N}\}\, .$$
	
	If $(d,M)=(-3,1)$, then $g=\pm A_{i_2}^{-1}(\pm {1\over 2}I+\(\begin{smallmatrix} 1/2 & 1\\ -1& -1/2\end{smallmatrix}\))A_{i_1}$ is in $\G^0(N)$ if and only if 
	\begin{equation}\label{conditions}
	N|uu'+vv'+ u'v\qquad\mbox{or}\qquad N|uu'+vv'+ uv'\, .\tag{$\star\star$}
	\end{equation}
	We will look for the solutions to the first condition, the latter one being symmetric.
	
	From similar reasons as in the previous situation, $u'=v\underline{u'}$ and\linebreak $\gcd (\underline{u'}, N)=1$. Hence, $\gcd (v,v')|uu'$, and thus $v$ and $v'$ are coprime. Our condition becomes ${N\over v} |v'+\underline{u'} (u+v)$ and implies $v'=\gcd (N,u+v)$. Let $t=u+v$ and write $t=v'\underline{t}$. It is easy to see that $t$ runs through the rests modulo $N/v$ and $\gcd(t,v)=1$. Moreover, ${N\over vv'} |1+\underline{u'}\underline{t}$.
	
	Hence, if we fix $v$ and $u$, $v'$ and $u'$ are uniquely determined. Similarly, if $N|u\tilde{u}+v\tilde{v}+ u\tilde{v}$, then $\tilde{u}$ and $\tilde{v}$ are uniquely determined by $u,v$. Moreover, it is easy to check that the conditions \eqref{conditions} hold at the same time only if $v=v'=1$ and $u=u'$ satisfies $u^2+u+1\equiv 0\pmod{N}$. Hence the conditions \eqref{conditions} and uniqueness of the solution for each of them imply that unless $v=1$ and $u^2+u+1\equiv 0\pmod{N}$, the matrix $\(\begin{smallmatrix} * & u\\ * & v\end{smallmatrix}\)$ is $\G^0(N)$-equivalent to exactly two matrices. Therefore, there are 
	$$\frac{N}{3}\prod_{p|N} (1+1/p) +2\frac{\mathcal{L}_{-3}}{3}$$
	$\G^0(N)$-non-equivalent classes within each class in $H(-3,L;\G^0(1))$, where
	$$\mathcal{L}_{-3} :=\#\{ u\in\{ 1,\ldots ,N\} :u^2+u+1=0\bmod{N}\}\, .$$
	
	In the following lemma we compute the quantities $\mathcal{L}_{-4}$ and $\mathcal{L}_{-3}$, and that finishes the proof.
\end{proof}
\begin{lemma} \mbox{}
	\begin{enumerate}
		\item Let $N=2^{n_0}p_1^{n_1}\ldots p_s^{n_s}$ be the prime decomposition of $N$. Then 
		$$\#\{ u\in\(\Z /N\Z\)\ti :u^2 =-1\bmod{N}\}
		=\begin{cases} 
		2^s & \mbox{if } n_0\leq1\mbox{ and }\,\forall_{i} p_i\equiv 1\, (\!\!\!\!\!\!\mod{4})\\
		0 & \mbox{otherwise.}\end{cases}\vspace{-0.2cm}	$$	
		\item Let $N=3^{n_0}p_1^{n_1}\ldots p_s^{n_s}$ be the prime decomposition of $N$. Then
		$$\#\{ u\in\(\Z /N\Z\)\ti :u^2+u+1=0\bmod{N}\}
		=\begin{cases} 
		2^s & \mbox{if } n_0\leq 1\mbox{ and }\,\forall_{i} p_i\equiv 1\, (\!\!\!\!\!\!\mod{6})\\
		0 & \mbox{otherwise.}\end{cases}$$
	\end{enumerate}
\end{lemma}
\begin{proof}
	This follows from Chinese Remainder Theorem and two basic facts: \begin{itemize}
		\item For an odd prime $p$, $\(\Z /p^{n}\Z\)\ti$ is a cyclic group of order $p^{n-1}(p-1)$.
		\item $\(\Z /2^n\Z\)\ti$ is cyclic of order $1$ and $2$ for $n=1$ and $2$, respectively. If $n\geq 3$, then it is the product of two cyclic groups, one of order $2$, the other of order $2^{n-2}$. 
	\end{itemize}
	\begin{enumerate}
		\item Depending whether $p\equiv 1\, (\!\!\!\!\mod{4})$ or $p\equiv -1\, (\!\!\!\!\mod{4})$, $-1$ is a square in $(\Z/p^n\Z)\ti$ or not, respectively. Moreover, because there is only one element of order $2$, there are either $2$ or $0$ solutions to $u^2 =-1\bmod{p^n}$.
		If $p=2$, $-1$ is a square in $(\Z/2^n\Z)\ti$ only if $n=1$, in which case $-1=1^2$.
		\item First note that the equation $u^2+u+1=0$ has no solution modulo $2$. Furthermore, because the solutions are of the form $(-1+\sqrt{-3})2^{-1}$, the equation has one solution modulo $3$ and no solutions modulo $3^n$ for $n>1$. Now, since $u^3-1=(u-1)(u^2+u+1=0)$, we will look for the elements of order $3$ in $(\Z/p^n\Z)\ti$, where $p\equiv\pm 1\, (\!\!\!\!\mod{6})$.
		
		If $p\equiv -1\, (\!\!\!\!\mod{6})$, then the order of $(\Z/p^n\Z)\ti$ is not divisible by $3$. In the other case, there are two elements of order $3$, $u_0$ and $u_0^2$, say. Both of them are zeros of the polynomial $u^2+u+1=0\bmod{p^n}$.
	\end{enumerate}
\end{proof}

In Proposition \ref{prop_G^0(N)_repr} we computed the size of $H(dM^2,L;\G^0(N))$. On the other hand, we know that (\cite{PSSc1}, proof of Proposition 5.3) 
$$|H_1(dM^2,L;\G^0(N))| = |\Cl_d(MN)| = {|\Cl_d(1)|\over u(d)} M N \prod_{p|MN} (1-p^{-1} \left(\frac{d}{p} \right))$$
if $MN>1$, where $u(d)$ is as in Proposition \ref{prop_G^0(N)_repr}. 
Combining these we get the following result.

\begin{prop}\label{phi_surjective} 
	The map $\tilde{\phi}_{L,M}\colon\Cl_d(MN)\to H(dM^2,L;\G^0(N))$
	is surjective if and only if $\left(\frac{dM^2}{p} \right)=-1$ for all primes $p|N$.
\end{prop} 

\begin{cor}\label{cor:H_1=H}
	The following conditions are equivalent:
	\begin{enumerate}
		\item $H_1(dM^2,L;\G^0(N))=H(dM^2,L;\G^0(N))$,
		\item $\left(\frac{dM^2}{p} \right)=-1$ for all $p|N$,
		\item $|H(dM^2N^2,L;\G^0(1))|=|H(dM^2,L;\G^0(N))|$.
	\end{enumerate}
\end{cor}
\begin{proof}
	This follows from the fact that $\Cl_d(MN)\cong H(dM^2N^2,L;\G^0(1))$ (Proposition 5.3, \cite{PSSc1}) and Proposition \ref{phi_injective}, \ref{phi_surjective}.
\end{proof}
\section{The relation}
Fix $\Phi\in V_{\pi}$, $N\in\N$ and assume that $\Phi$ is right invariant by the subgroup $I_{N_1,N_2}$ of $G(\A_f)$ for some $N_1|N_2$. 

For any two integers $L,M$, we define the element $H(L,M)\in G(\A)$ by
$$H(L,M)_p:=\begin{cases} \(\begin{smallmatrix} LM^2 & & & \\ & LM & & \\ & & 1 & \\ & & & M \end{smallmatrix}\) ,& p\mid LM\\ \(\begin{smallmatrix} 1 & & & \\ & 1 & & \\ & & 1 & \\ & & & 1 \end{smallmatrix}\) ,& p\nmid LM\mbox{ or } p=\infty\end{cases}\, .$$
Note that $H(1,1) = 1$ and $\Phi(gH(L,M)_p) = \Phi(gh_p(l_p, m_p))$ for all $p<\infty$, $g\in G(\A)$, where $l_p=\ord_p L$ and $m_p=\ord_p M$. For the rest of this section we relax the notation and abbreviate $h_p(l, m):=h_p(l_p, m_p)$.

Fix a fundamental discriminant $d<0$ and put $S=S(d)$, $T=T_S$. Let $\psi = \prod_v\psi_v$ be the character of $\A$ such that
\begin{itemize}
	\item the conductor of $\psi_p$ is $\Z_p$ for all (finite) primes $p$,
	\item $\psi_\infty(x) = e(x),$ for $x \in \R$,
	\item $\psi|_\Q =1;$
\end{itemize}
so that $\theta(u(X)):=\psi(\tr(SX))$ is a character of $U(\A)$. Let $\Lambda= \prod_{p\leq \infty}\Lambda_p$ be a character of $T(\A )/T(\Q )T(\R )$ such that $\Lambda|_{\A^\times} = 1$, and let $C(\Lambda)=\prod_p p^{c(\Lambda_p)}$ be the smallest integer such that $\Lambda|_{T_{C(\Lambda)}}=1$ (cf. Def. \ref{def:c(Lambda)}, \ref{def:T_N}).
Suppose that $\Phi = \tilde{\Phi}\otimes_{p\nmid N_2} \phi_p$ is a pure tensor in the space of $\pi$ away from the level. If $\pi$ has a global Bessel model of type $(\Lambda, \theta)$, then for each place $p$ of $\Q$,  $\pi_p$ has a local Bessel model of type $(\Lambda_p, \theta_p)$ and each $\phi_p$ corresponds to a (unique up to multiples) vector $B_{\phi_p}$ in the local Bessel model of $\pi_p$.
\begin{dfn}
	For each $\Phi\in V_{\pi}$ and a character $\Lambda$ such that a global $(\Lambda,\theta)$-Bessel model exists, define
	$$M_{\Phi, \Lambda}:=\prod_{p\nmid N_2} p^{c(\Lambda_p)}.$$
\end{dfn}

\begin{rem}
	If $\pi_p$ is a spherical representation with trivial central character and it is not of type I (that is, it is of type IIb, IIIb, IVd, Vd or VId), then $\pi_p$ admits a local Bessel model if and only if $\Lambda=\triv$, in which case $c(\Lambda_p)=0$. The only representations that will occur in the proof of main theorem regarding Maass relations are of type IIb, and thus in what follows we could take simply $M_{\Phi, \Lambda}=1$.
\end{rem}

The following lemma is the base for our main results.

\begin{lemma}\label{lem:Besselrelation} 
Let $L,M, L', M'$ be positive integers such that $L'|L$, $M_{\Phi, \Lambda} |M'|M$, $(L,N_2^{\infty})=(L',N_2^{\infty})$ and $(M,N_2^{\infty})=(M',N_2^{\infty})$. If a local $(\Lambda_p, \theta_p)$-Bessel model exists at all primes $p\nmid N_2$, then the following relation holds:
\begin{equation}\label{Besselrelation} 
B_\Phi (H(L,M))\prod_{\substack{p|L'M'\\ p\nmid N_2}} B_{\phi_p}(h_p(l',m')) = B_\Phi (H(L',M'))\prod_{\substack{p|LM\\ p\nmid N_2}} B_{\phi_p}(h_p(l,m))\, .
\end{equation}
Otherwise, $B_\Phi (H(\underline{L},\underline{M}))=0$ for all integers $\underline{L},\underline{M}$.
\end{lemma}
\begin{proof}
Observe first that if $\pi$ does not have a global Bessel model of type $(\Lambda, \theta)$, then both sides of \eqref{Besselrelation} are zero. We may assume then that the global Bessel model exists.

By uniqueness of local Bessel functionals, 
$$B_{\Phi}(H(L,M))=C\prod_{p|N_2} B_{\phi_p}(h_p(l,m))\prod_{p\nmid N_2} B_{\phi_p}(h_p(l,m))$$
for any positive integers $L,M$. The constant $C$ can be found if we specialise to $L_{\S}:=(L,N_2^{\infty})$ and $M_{\Phi, \Lambda}M_{\S}$, where $M_{\S}:=(M,N_2^{\infty})$: 
$$B_{\Phi}(H(L_{\S}, M_{\Phi, \Lambda}M_{\S}))=C\prod_{p|N_2} B_{\phi_p}(h_p(l,m))$$
(recall that we set the normalisation $B_{\phi_p}(h_p(0,\ord_p(M_{\Phi, \Lambda})))=1$ for $p\nmid N_2$), and thus
$$B_{\Phi}(H(L,M))=B_{\Phi}(H(L_{\S}, M_{\Phi, \Lambda}M_{\S}))\prod_{\substack{p|LM\\ p\nmid N_2}} B_{\phi_p}(h_p(l,m))\, .$$
Using this equation with $L', M'$ as in the statement of the lemma, we obtain the relation \eqref{Besselrelation}. Note that without the assumption $M_{\Phi, \Lambda}|M'$ the statement is still true, but we have zeros on both sides of the equality.
\end{proof}

From the simple looking relation \eqref{Besselrelation} we obtain a correspondence between the Fourier coefficients \eqref{fouriercoefficientdefeq} that will be crucial for our applications. We start with an auxiliary lemma.

\begin{lemma}\label{lem:auxiliary}
Let $A\in \GL_2(\Q), \alpha\in\Q\ti, \gamma =\(\begin{smallmatrix} A & \\ & \alpha\T{A}^{-1} \end{smallmatrix}\)$, and let $\gamma_f$ be the image of $\gamma$ in $G(\A_f)$. Then for any automorphic form $\Phi$ on $G(\A)$, any matrix $T\in M_2^{\rm sym}(\Q)$ and $g_{\infty}\in G(\R)^+$ we have
$$\Phi_T(g_{\infty}\gamma_f)=\Phi_{\alpha^{-1}\T{A}TA}(\gamma_{\infty}^{-1} g_{\infty})\, ,$$
where $\gamma_{\infty}=\gamma\gamma_f^{-1}$.
\end{lemma} 
\begin{proof}
Using the fact that $\Phi$ is left $G(\Q)$-invariant and the substitution\linebreak $X=\alpha^{-1}AY\T{A}$, we obtain
\begin{align*}
\Phi_T(g_{\infty}\gamma_f) &=\int_{M_2^{\rm sym}(\Q)\backslash M_2^{\rm sym}(\A)}\psi^{-1}({\tr}(TX))\Phi(u(X)g_{\infty}\gamma_f) dX\\
&=\int_{M_2^{\rm sym}(\Q)\backslash M_2^{\rm sym}(\A)}\psi^{-1}({\tr}(TX))\Phi(u(X)\gamma\gamma_{\infty}^{-1} g_{\infty}) dX\\
&=\int_{M_2^{\rm sym}(\Q)\backslash M_2^{\rm sym}(\A)}\psi^{-1}({\tr}(TX))\Phi(u(\alpha A^{-1}X\T{A}^{-1})\gamma_{\infty}^{-1} g_{\infty}) dX\\
&=\Phi_{\alpha^{-1}\T{A}TA}(\gamma_{\infty}^{-1} g_{\infty})\, .
\end{align*}
\end{proof}

\begin{lemma}\label{lem:Phi-to-a(F,T)}
Let $\Phi$ be an automorphic form on $G(\A)$ that satisfies the equation \eqref{eq:weight}, and let $F$ be as in \eqref{deadelisation}. Then for any matrix $T\in M_2^{\rm sym}(\Q)$ and $g_{\infty}\in G(\R)^+$ we have
$$\Phi_T(g_{\infty})=\mu(g_{\infty})^k j(g_{\infty},i_2)^{-k} a(F,T) e(\tr (T(g_{\infty}\cdot i_2)))\, .$$
\end{lemma} 
\begin{proof}
It is easy to check that the automorphy factor $j$ has the property
$$j(g_1g_2, Z)=j(g_1,g_2\cdot Z)j(g_2,Z)\, .$$
Hence,
\begin{align*}
\Phi_T(g_{\infty})&=\int_{M_2^{\rm sym}(\Q)\backslash M_2^{\rm sym}(\A)}\psi^{-1}({\tr}(TX))\Phi(u(X)g_{\infty}) dX\\
&=\int_{M_2^{\rm sym}(\Z)\backslash M_2^{\rm sym}(\R)} e(-\tr(TX))\mu(g_{\infty})^k j(u(X)g_{\infty},i_2)^{-k} F(u(X)g_{\infty}\cdot i_2) dX\\
&=\mu(g_{\infty})^k j(g_{\infty},i_2)^{-k} \sum_{T'}a(F,T')\\
&\hspace{0.4cm}\int_{M_2^{\rm sym}(\Z)\backslash M_2^{\rm sym}(\R)} e(-\tr(TX)) e(\tr (T'(g_{\infty}\cdot i_2))) e(\tr (T'X)) dX\\
&=\mu(g_{\infty})^k j(g_{\infty},i_2)^{-k} a(F,T)e(\tr (T(g_{\infty}\cdot i_2)))\, .
\end{align*}
\end{proof}

\begin{thm}\label{thm:the_relation}
Let $\pi$ be an irreducible automorphic cuspidal representation of $G(\A)$ with trivial central character and $\Phi\in V_\pi$ an automorphic form in its vector space. Assume that $\Phi$ is right $I_{N_1,N_2}$-invariant for some $N_1|N_2$. Let $S=S(d)$, and $\psi$ as above. Let $\Lambda$ be a character of $T(\A )/T(\Q )T(\R )$ such that $\Lambda|_{\A^\times} = 1$. Then for any $L,M,L',M'$ satisfying the conditions of Lemma \ref{lem:Besselrelation} and such that $C(\Lambda)|M'N_1$, we have the following correspondence between the Fourier coefficients of $F$:
\begin{multline}\label{the_relation}
\hsp{0.3}\frac{|\Cl_d(M'N_1)|}{|\Cl_d(MN_1)|}\sum_{T\in H_1(dM^2,L;\G^0(N_1))}\hsp{0.1}\Lambda^{-1}(T) a(F,T) \prod_{\substack{p\mid L'M'\\ p\nmid N_2}} B_{\phi_p}(h_p(l',m'))\\ 
=\sum_{T'\in H_1(dM'^2,L';\G^0(N_1))}\hsp{0.1}\Lambda^{-1}(T')  a(F,T') \prod_{\substack{p\mid LM\\ p\nmid N_2}} B_{\phi_p}(h_p(l,m)).
\end{multline} 
\end{thm}
\begin{proof}
In view of the relation \eqref{Besselrelation}, it suffices to compute the values $B_\Phi (H(L,M))$.

Let $\Phi^{L,M}(g):=\Phi(gH(L,M))$ for $g\in G(\A)$; because $\Phi$ is right $I_{N_1,N_2}$-invariant, $\Phi^{L,M}$ is right invariant by
$$H_{\infty} I_{N_1,N_2} H_{\infty}^{-1}=\{\prod_{p<\infty}\(\begin{smallmatrix} * & MN_1* & LM^2* & LM*\\ */M & * & LM* & L*\\ *N_2/LM^2 & *N_2/LM & * & */M\\ *N_2/LM & *N_2/L & MN_1* & * \end{smallmatrix}\)\in G(\Q_p ):*\in\Z_p\} ,$$
where $H_{\infty}:=\(\begin{smallmatrix} LM^2 & & & \\ & LM & & \\ & & 1 & \\ & & & M \end{smallmatrix}\)\in G(\R )^+$. In particular, $\Phi^{L,M}$ is right invariant by $T_{MN_1}$ and $K^*_{MN_1}$.

Following the notation of section \ref{sec:ray_class_groups}, we can write  
$$T(\A)=\coprod_{c\in\Cl_d(MN_1)} t_c T(\Q)T(\R)T_{MN_1},$$
where we choose $t_c\in\prod_{p<\infty} T(\Q_p)$, and by strong approximation theorem write $t_c=\gamma_c m_c\kappa_c$ with $\gamma_c\in\GL_2(\Q)$, $m_c\in\GL_2(\R)^+$ and $\kappa_c\in K^*_{MN_1}$.

With this preparation we are ready to compute the values $B_{\Phi}(H(L,M))$.
\begin{multline*}
B_{\Phi}(H(L,M)) 
= \int_{\A\ti T(\Q )\back T(\A )}\hsp{0.5}\Lambda^{-1}(t)\Phi_S^{L,M}(t)dt\\
=\hsp{0.2} \sum_{c\in\Cl_d(MN_1)} \int_{\R\ti\back T(\R )} \Phi_S^{L,M}(t_ct_{\infty}) dt_{\infty} \int_{\A_f\ti T(\Q )\cap T_{MN_1}\back T_{MN_1}}\hsp{0.5} \Lambda^{-1}(t_ct_{\infty}t_{MN_1}) dt_{MN_1}
\end{multline*}
Observe that if $C(\Lambda)\nmid MN_1$, then the inner integral is equal to zero and the equation \eqref{the_relation} holds. Henceforth we assume that $C(\Lambda)|MN_1$. With this assumption and using Lemma \ref{lem:auxiliary} twice, we have
\begin{align*}
B_{\Phi}(H(L,M)) &=\frac{1}{|\Cl_d(MN_1)|}\sum_{c\in\Cl_d(MN_1)}\hsp{0.4} \Lambda^{-1}(t_c)\int_{\R\ti\back T(\R )}\hsp{0.4} \Phi_S^{L,M}(\gamma_c m_ct_{\infty}) dt_{\infty}\\
&= \frac{1}{|\Cl_d(MN_1)|}\sum_{c\in\Cl_d(MN_1)}\hsp{0.4} \Lambda^{-1}(t_c)\int_{\R\ti\back T(\R )}\hsp{0.4} \Phi_S^{L,M}(t_{\infty}(\gamma_c)_f)dt_{\infty}\\
&= \frac{1}{|\Cl_d(MN_1)|}\sum_{c\in\Cl_d(MN_1)}\hsp{0.4} \Lambda^{-1}(t_c)\int_{\R\ti\back T(\R )}\hsp{0.4} \Phi_{S_c}^{L,M}(m_ct_{\infty})dt_{\infty}\\
&= \frac{1}{|\Cl_d(MN_1)|}\sum_{c\in\Cl_d(MN_1)}\hsp{0.4} \Lambda^{-1}(t_c)\int_{\R\ti\back T(\R )}\hsp{0.4} \Phi_{\phi_{L,M}(c)} (H_{\infty}^{-1}m_ct_{\infty})dt_{\infty}\, ,
\end{align*}
where $\phi_{L,M}(c)$ is defined as in \eqref{def:phi_{L,M}}. Further,
\begin{align*}
\int_{\R\ti\back T(\R )} \Phi_{\phi_{L,M}(c)}(H_{\infty}^{-1} m_ct_{\infty}) dt_{\infty}\\
&\hspace{-4.5cm}\stackrel{\rm{Lemma}\, \ref{lem:Phi-to-a(F,T)}}{=}(LM)^{-k} a(F, \phi_{L,M}(c))\\
&\hspace{-2.7cm} \cdot\int_{\R\ti\back T(\R )}e(\tr (\phi_{L,M}(c) (H_{\infty}^{-1}m_ct_{\infty}\cdot i_2))) dt_{\infty}\\
&\hspace{-3.9cm} =(LM)^{-k} a(F, \phi_{L,M}(c)) \int_{\R\ti\back T(\R )} e(\tr (S_c (m_ct_{\infty}\cdot i_2))) dt_{\infty}\\
&\hspace{-3.9cm} =(LM)^{-k} a(F,\phi_{L,M}(c)) \int_{\R\ti\back T(\R )} e(\tr (Si_2))dt_{\infty}\\
&\hspace{-3.9cm} =r(LM)^{-k} a(F,\phi_{L,M}(c)) e^{-2\pi\tr S} ,
\end{align*}
where $r=\int_{\R\ti\back T(\R )} dt_{\infty}$. Finally, recall that 
$$\{\phi_{L,M}(c)\colon c\in\Cl_d(MN_1)\} =H_1(dM^2,L;\G^0(N_1)) .$$
\end{proof}

\begin{cor}\label{cor:the_relation_easy}
	Let $F$ be a Siegel cusp form of degree $2$, level $\G_0(N_1,N_2)$ and weight $k$. Suppose that $F$ is an eigenform of the local Hecke algebra at all primes $p\nmid N_2$\footnote{It is enough to assume that $F$ is an eigenform of the Hecke operators $T(p)$ and $T(p^2)$ at $p\nmid N_2$. For the definition of these Hecke operators see for example \cite{andzhu}.}. Let $d$ be a fundamental discriminant and let $L,M,L',M'$ be positive integers such that $L'|L$, $M'|M$, $(L,N_2^{\infty})=(L',N_2^{\infty})$ and $(M,N_2^{\infty})=(M',N_2^{\infty})$. Then for all characters $\Lambda$  of $H_1(dM'^2,L,\G^0(N_1))\cong\Cl_d(M'N_1)$,
	\begin{multline}\label{eq:the_relation_simple}
	\hsp{0.2}\frac{|\Cl_d(M'N_1)|}{|\Cl_d(MN_1)|}\(\frac{L'M'}{LM}\)^k\sum_{T\in H_1(dM^2,L;\G^0(N_1))}\hsp{0.3}\Lambda^{-1}(T) a(F,T) \prod_{\substack{p\mid L'M'\\ p\nmid N_2}} B_{\phi_p}(h_p(l',m'))\\ 
	=\sum_{T'\in H_1(dM'^2,L';\G^0(N_1))}\hsp{0.3}\Lambda^{-1}(T')  a(F,T') \prod_{\substack{p\mid LM\\ p\nmid N_2}} B_{\phi_p}(h_p(l,m))\, .
	\end{multline} 
\end{cor}
\begin{proof}
	This follows immediately from Theorem \ref{thm:the_relation} and Proposition 3.11, \cite{sahapet}, which states that in our setting the following conditions are equivalent:
	\begin{itemize}
		\item[(i)] $F$ is an eigenform of the local Hecke algebra at all primes $p\nmid N_2$.
		\item[(ii)] If $\pi',\pi''$ are two irreducible cuspidal representations both of which occur as subrepresentations of the representation $\pi$ associated with $F$, then $\pi_p'\cong\pi_p''$ for all primes $p\nmid N_2$.
	\end{itemize}
	As a result, $F=\sum_i F_i$, where each $F_i$ has the same local data at $p\nmid N_2$ and is as in Theorem \ref{thm:the_relation}.
\end{proof}
\section{Maass relations}
Let $F$ be a cuspidal Siegel modular form invariant under the action of $\G_0(N_1,N_2)$ that is an eigenform of the local Hecke algebra at primes $p\nmid N_2$. Let $\Phi$ be the adelisation of $F$, and $\pi=\otimes_p\pi_p$ the associated automorphic representation. Suppose that for primes $p\nmid N_2$, $\pi_p=\chi_p\triv_{\GL(2)}\rtimes\chi_p^{-1}$ with an unramified character $\chi_p$ of $\Q_p^\times$ (a representation of type IIb according to Table \ref{table2}). Note that these are non-tempered, non-generic representations. 
The set of $\pi$ obtained in this way is precisely the set of CAP representations attached to the Siegel parabolic subgroup of $G(\A)$ (cf. \cite{explicit_bocherer}).

\begin{lemma}\label{T_invariance}
For representation $\pi$ as above, any vector $\tilde\Phi =\otimes_p\tilde\phi_p$ in the vector space $V_{\pi}$ of $\pi$ and any non-degenerate matrix $S\in M_2^{\rm sym}(\Q)$, we have:
$$\tilde\Phi_S(tg)=\tilde\Phi_S(g)\quad\mbox{for all}\quad g\in G(\A)\quad\mbox{and}\quad t\in\prod_{p\nmid N_2} T_S(\Q_p)\prod_{p|N_2} 1_4\, .$$
\end{lemma}
\begin{proof}
Let $\tilde\Phi=\otimes_p\tilde\phi_p$ be as in the lemma, and let $\S=\{ p:p|N_2\}$. Without loss of generality we may assume $g=1_4$. Let $V^{\S}$ be the subspace of $V_{\pi}$ generated by all vectors of the form $\otimes_{p\in\S}\tilde\phi_p\otimes_{p\notin\S}\psi_p$ with $\psi_p\in V_{\pi_p}$. The right action of $\otimes_{p\notin\S} G(\Q_p)$ on $V^{\S}$ makes $V^{\S}$ a representation isomorphic to $\otimes_{p\notin\S}\pi_p$ (we put here $\Q_{\infty}:=\R$). Define 
$$\beta\colon V^{\S}\to\C\, ,\qquad \beta(\Psi):=\Psi_S(1_4)=\int_{U(\Q )\back U(\A)} \Psi(u)\theta_S^{-1}(u) du\, .$$ 
Note that
$\beta(\pi(t)\Psi)=\Psi_S(t)$. We need to show that $\beta(\pi(t)\tilde\Phi)=\beta(\tilde\Phi)$ for all $t\in\prod_{p\notin\S} T(\Q_p)$. This is trivial if $\beta\equiv 0$. So assume $\beta\not\equiv 0$. Let 
$$\Phi'=\otimes_{p\in\S}\tilde\phi_p\otimes_{p\notin\S}\phi_p'\quad\mbox{ be such that } \quad\beta(\Phi')\neq 0\, .$$ 
For each place $p\notin\S$ we get a functional $\beta_p$ on $V_{\pi_p}$ via 
$$\beta_p(\psi_p):=\beta(\psi_p\otimes_{q\in\S}\tilde\phi_q\otimes_{q\notin\S\cup\{ p\}} \phi_q')\, .$$ 
Then $\beta_p(\phi_p')\neq 0$ and thus $\beta_p$ is a non-zero functional on $V_{\pi_p}$. Clearly, $\beta_p$ satisfies
$$\beta_p(\pi_p(u)\psi_p)=\theta_S(u)\beta_p(\psi_p)\qquad\mbox{for all}\quad \psi_p\in V_{\pi_p}\mbox{ and } u\in U(\Q_p)$$
By \cite[Corollary 4.2]{pitsch2}, the matrix $S$ satisfies the conditions of \cite[Lemma 4.1]{PSSc1}, and therefore by this lemma
\begin{itemize}
\item the space of such functionals $\beta_p$ is one-dimensional,
\item $\beta_p(\pi_p(t)\psi_p)=\beta_p(\psi_p)$ for all $\psi_p\in V_{\pi_p}$ and $t\in T(\Q_p)$.
\end{itemize}
So there exists a constant $C_{\S}$ such that 
$$\beta(\Psi)=C_{\S}\prod_{p\notin\S} \beta_p(\psi_p)$$
whenever $\Psi\in V^{\S}$ corresponds to $\otimes_{p\in\S}\tilde\phi_p\otimes_{p\notin\S}\psi_p$. Hence $\beta(\pi(t)\tilde\Phi)=\beta(\tilde\Phi)$ for all $t\in\prod_{p\notin\S} T(\Q_p)$.
\end{proof}

\begin{lemma}
Let $F, N_1$ be as above, $S=S(d)$, and $L,M$ any positive integers. Then for any $c_1,c_2\in\Cl_d(MN_1)$,
$$a(F,L\(\begin{smallmatrix} M &  \\  & 1\\\end{smallmatrix}\)S_{c_1}\(\begin{smallmatrix} M &  \\  & 1\\\end{smallmatrix}\) )=a(F,L\(\begin{smallmatrix} M &  \\  & 1\\\end{smallmatrix}\)S_{c_2}\(\begin{smallmatrix} M &  \\  & 1\\\end{smallmatrix}\) )\, .$$
\end{lemma}
\begin{proof}
Let $\{ t_c\}_c$ be a set of representatives of $\Cl_d(MN_1)$. We may choose $t_c$ so that $t_{c,p}=1_2$ for all $p|N_2$. Indeed, if $\tilde{t}\in T(\Q)$ is such that $\tilde t_p=t_{c,p}$ for all $p|N_2$, then $t_c=\tilde{t}\prod_{p|N_2} 1_2\prod_{p\nmid N_2} \tilde{t}_p^{-1}t_{c,p}$.
From the proof of Theorem \ref{thm:the_relation}, and using its notation, we get 
$$a(F,L\(\begin{smallmatrix} M &  \\  & 1\\\end{smallmatrix}\)S_c\(\begin{smallmatrix} M &  \\  & 1\\\end{smallmatrix}\) )= (LM)^k e^{2\pi\tr S} {1\over r}\int_{\R\ti\back T(\R )}\Phi_{\phi_{L,M}(c)}(H_{\infty}^{-1}m_c t_{\infty})dt_{\infty}\, ,$$
where
$$\Phi_{\phi_{L,M}(c)}(H_{\infty}^{-1}m_c t_{\infty})=\Phi_S(t_ct_{\infty}H(L,M))\stackrel{\rm{Lemma}\, \ref{T_invariance}}{=} \Phi_S(t_{\infty}H(L,M))$$
does not depend on $c$.
\end{proof}

Hence, it makes sense to write $a(F;dM^2,L)$ for any Fourier coefficient of $F$ that is of the form $a(F,L\(\begin{smallmatrix} M &  \\  & 1\\\end{smallmatrix}\)S_c\(\begin{smallmatrix} M &  \\  & 1\\\end{smallmatrix}\) )$ for some $c\in\Cl_d(MN_1)$, or in another words, for $a(F,T)$ with $T\in H_1(dM^2,L;\G^0(N_1))$.

The following theorem generalises \cite[Theorem 5.1]{PSSc1} to cuspidal Siegel modular forms of level $\G_0(N_1,N_2)$ with $N_2>1$.

\begin{thm}\label{thm:Maass_relation}
Let $F$ be as above. For any fundamental discriminant $d$ and any positive integers $L,M$, Fourier coefficients of $F$ satisfy the following Maass relations:
\begin{equation}\label{eq:Maass_relation}
a(F;dM^2,L)=\sum_{\substack{r|L\\\gcd(r,N_2)=1}} r^{k-1} a(F;d\(\frac{LM}{rL_{\S}}\)^2,L_{\S}),
\end{equation}
where $L_{\S}=(L, N_2^{\infty})$.

Hence, if 
\begin{equation}\label{special_coeff}
\mat{a}{b/2}{b/2}{c} =L\mat{M}{}{}{1} S_{\tilde{c}}\mat{M}{}{}{1}
\end{equation}
for some $\tilde{c}\in\Cl_d(MN_1)$, then
$$a(F,\mat{a}{b/2}{b/2}{c}) =\sum_{\substack{r|\gcd(a,b,c)\\\gcd(r,N_2)=1}} r^{k-1} a(F,L_{\S}\mat{ac\over (rL_{\S})^2}{b\over 2rL_{\S}}{b\over 2rL_{\S}}{1})\, ,$$
where $L_{\S}=(\gcd(a,b,c), N_2^{\infty})$.
\end{thm}
\begin{proof}
Recall Corollary \ref{cor:the_relation_easy}. If we take $\Lambda$ to be a trivial character\footnote{Since at $p\nmid N_2$, $\pi_p=\chi_p\triv_{\GL(2)}\rtimes\chi_p^{-1}$ is a representation of type IIb, so according to Table \ref{table2} a local $(\Lambda_p, \theta_p)$-Bessel model exists if and only if $\Lambda_p =\triv$.}
in the formula \eqref{eq:the_relation_simple}, and $M'=(M,N_2^{\infty})$, $L'=(L,N_2^{\infty})$, then
$$a(F;dM^2,L)=\(\frac{LM}{L'M'}\)^k a(F;dM'^2,L')\prod_{\substack{p\mid LM\\ p\nmid N_2}} B_{\phi_p}(h_p(l_p, m_p))\, ,$$
where $l_p=\ord_p L$, $m_p=\ord_p M$. Similarly, for any divisor $r$ of $L$,
$$a(F;d\(\frac{LM}{rL'}\)^2\hsp{0.2},L')=\hsp{0.1}\(\frac{LM}{rL'M'}\)^k\hsp{0.1} a(F;dM'^2,L')\hsp{0.1}\prod_{\substack{p\mid LM/r\\ p\nmid N_2}}\hsp{0.2} B_{\phi_p}(h_p(0, l_p+m_p-r_p))\, .$$
Note that the last product can actually be taken over all primes $p|LM$ that do not divide $N_2$. Indeed, in the product over $p|LM/r, p\nmid N_2$ we miss only those places $p$ for which both $r_p=l_p$ and $m_p=0$. But in this case $B_{\phi_p}(h_p(0, l_p+m_p-r_p))=B_{\phi_p}(h_p(0,0))=1$ by Theorem \ref{thm:Sugano}.

Moreover, it is known \cite[Theorem 2.1]{PSSc1} that the spherical vectors of the representations of type IIb satisfy the equation
$$B_p(h(l,m))=\sum_{i=0}^l p^{-i}B_p(h(0,l+m-i))$$
for all $l,m\geq 0$.
Hence, the equation \eqref{eq:Maass_relation} holds if and only if
$$\prod_{\substack{p\mid LM\\ p\nmid N_2}}\sum_{i=0}^{l_p} p^{-i}B_{\phi_p}(h_p(0, l_p+m_p-i))=\sum_{\substack{r|L\\ \gcd(r,N_2)=1}}{1\over r}\prod_{\substack{p\mid LM\\ p\nmid N_2}} B_{\phi_p}(h_p(0, l_p+m_p-r_p))\, ,$$
which is true.
\end{proof}

It is unfortunate that our method gives us an access only to the coefficients of the form \eqref{special_coeff}, which are not yet fully characterised. Nevertheless, their partial study in section \ref{sec:ray_class_groups} with subsequent Corollary \ref{cor:H_1=H} already leads to a satisfactory result which generalises the full level case:

\begin{cor}
Let $F$, $N_1$, $N_2$ be as above. Then for any matrix $T=L\(\begin{smallmatrix} a & b/2 \\ b/2 & c\end{smallmatrix}\)$ for which $L|N_2^{\infty}$, $\gcd(a,b,c,N_2)=1$ and $\(\frac{b^2-4ac}{p}\) =-1$ for every $p|N_1$, we have 
$$a(F, L\mat{a}{b/2}{b/2}{c}) = \sum_{r|\gcd(a,b,c)} r^{k-1}  a(F, L\mat{{ac\over r^2}}{{b\over 2r}}{{b\over 2r}}{1}).$$
\end{cor}


\end{document}